\documentclass[12pt]{amsart}

\usepackage[utf8]{inputenc}
\usepackage{amsfonts}
\usepackage{amsmath}
\usepackage{graphicx}
\usepackage{float}
\usepackage[dvipsnames]{xcolor}
\usepackage{hyperref}
\usepackage{tikz-cd}

\usepackage{amssymb}
\usepackage{enumitem}
\usepackage{mathrsfs}

\usepackage{tikz}
\usetikzlibrary{topaths}
\usetikzlibrary{calc}

\usepackage{a4wide}

\usepackage{empheq}

\newtheorem{theorem}{Theorem}[section]
\newtheorem*{theorem*}{Theorem}
\newtheorem{lemma}[theorem]{Lemma}
\newtheorem{proposition}[theorem]{Proposition}
\newtheorem*{proposition*}{Proposition}
\newtheorem{corollary}[theorem]{Corollary}
\newtheorem*{corollary*}{Corollary}
\newtheorem{conjecture}[theorem]{Conjecture}
\newtheorem{cit}[theorem]{Citation}
\newtheorem*{conjecture*}{Conjecture}
\newtheorem{question}[theorem]{Question}
\newtheorem*{question*}{Question}

\newtheorem{lettertheorem}{Theorem}

\theoremstyle{definition}
\newtheorem{definition}[theorem]{Definition}
\newtheorem*{definition*}{Definition}
\newtheorem{remark}[theorem]{Remark}
\newtheorem{example}[theorem]{Example}
\newtheorem{observation}[theorem]{Observation}

\newcommand{\vN}{\mathcal{N}}

\newcommand{\N}{\mathbb{N}}
\newcommand{\Z}{\mathbb{Z}}
\newcommand{\Q}{\mathbb{Q}}

\DeclareMathOperator{\Hom}{Hom}
\DeclareMathOperator{\Aut}{Aut}

\DeclareMathOperator{\F}{\mathcal{F}}
\DeclareMathOperator{\FP}{\mathcal{FP}}

\DeclareMathOperator{\Stab}{Stab}
\DeclareMathOperator{\RStab}{RStab}
\DeclareMathOperator{\LSupp}{LSupp}

\DeclareMathOperator{\id}{id}

\DeclareMathOperator{\gtwist}{gtwist}
\DeclareMathOperator{\A}{A}
\DeclareMathOperator{\HA}{HA}

\numberwithin{equation}{section}

\begin{document}

\title{Abstract twisted Brin--Thompson groups}
\date{\today}


\author[F.~Fournier-Facio]{Francesco Fournier-Facio}
\address{Department of Pure Mathematics and Mathematical Statistics, University of Cambridge, UK}
\email{ff373@cam.ac.uk}

\author[X.~Wu]{Xiaolei Wu}
\address{Shanghai Center for Mathematical Sciences, Jiangwan Campus, Fudan University, No. 2005 Songhu Road, Shanghai, 200438, P.R. China}
\email{xiaoleiwu@fudan.edu.cn}

\author[M.~C.~B.~Zaremsky]{Matthew C.~B.~Zaremsky}
\address{Department of Mathematics and Statistics, University at Albany (SUNY), Albany, NY}
\email{mzaremsky@albany.edu}

\begin{abstract}
Given a group $G$ acting faithfully on a set $S$, one gets a simple group denoted $SV_G$, called a twisted Brin--Thompson group. In this paper we drop the faithfulness assumption, and get an abstract version of a twisted Brin--Thompson group $SV_G$. While the resulting group is not simple, since $SV_G$ surjects onto $SV_{G/\ker(G \curvearrowright S)}$, we prove that every proper normal subgroup of $SV_G$ lies in the kernel of this surjection, so $SV_G$ is ``relatively simple''. The advantage is that now we can prove that every finitely presented simple group embeds in a finitely presented abstract twisted Brin--Thompson group intersecting this kernel trivially. In particular, if the Boone--Higman conjecture is true, then so is a related conjectural characterization of groups with solvable word problem, arising purely in the world of twisted Brin--Thompson groups. We also prove a variety of additional results about abstract twisted Brin--Thompson groups, some of which are new even in the faithful case: they are all uniformly perfect, have property NL and property FW$_\infty$, are boundedly acyclic and $\ell^2$-invisible, and are $C^*$-simple as soon as they have trivial amenable radical. Along the way we formulate a new general criterion for $\ell^2$-invisibility that is interesting in its own right.
\end{abstract}

\maketitle
\thispagestyle{empty}

\section{Introduction}\label{sec:intro}

Twisted Brin--Thompson groups, introduced by Belk and the third author \cite{belk22} following the non-twisted case due to Brin \cite{brin04}, are a family of simple groups that have emerged recently as having outsized importance. For example, with Belk and Hyde the first and third authors proved that every finitely presented simple group that is mixed identity-free embeds in a finitely presented twisted Brin--Thompson group \cite{BFFHZ}. Also, twisted Brin--Thompson groups reveal that every finitely generated group quasi-isometrically embeds in a finitely generated simple group \cite[Corollary~C]{belk22}.

The input to constructing a twisted Brin--Thompson group $SV_G$ is a faithful action of a group $G$ on a set $S$. If $C$ denotes the binary Cantor set with the defining action of Thompson's group $V$, there is a natural action of the permutational wreath product $V \wr_S G$ on $C^S$. Then $SV_G$ can be defined as the topological full group of this action. Despite this definition as a group of homeomorphisms, a large portion of the toolbox to study twisted Brin--Thompson groups is purely combinatorial. Over the years, it has become clear that, from this combinatorial point of view, the requirement of faithfulness is both frustrating and also not entirely necessary for some purposes.

The goal of this paper is therefore to introduce ``abstract'' twisted Brin--Thompson groups $SV_G$ for $G\curvearrowright S$ not necessarily faithful (Definition~\ref{def:main}), analyze the properties that remain true in the non-faithful case, and deduce some relations to the Boone--Higman conjecture. In case $S$ is a point, this recovers the labeled Thompson group $V(G)$, which was essentially introduced by Thompson in \cite{thompson80}, see also \cite[Subsection~1.5]{wu25} for a more modern treatment.

\subsection{Relative simplicity} While our abstract twisted Brin--Thompson groups are not simple, outside the faithful case, they are what we call relatively simple, as we now explain.

\begin{definition}[Normal pair]
By a \emph{normal pair} $(G,N)$ we mean a group $G$ together with a normal subgroup $N\le G$. Call $(G,N)$ \emph{proper} if $N$ is proper in $G$. Call $(G,N)$ \emph{finitely generated/presented} if $G$ is, with no assumption on $N$.
\end{definition}

\begin{definition}[Relatively simple]
Call a proper normal pair $(G,N)$ \emph{relatively simple} if every proper normal subgroup of $G$ is contained in $N$. Call the group $G$ \emph{relatively simple} if $(G,N)$ is relatively simple for some proper normal subgroup $N$, and in this case call $N$ the \emph{largest normal subgroup} of $G$.
\end{definition}

If $(G,N)$ is relatively simple then $G/N$ is simple. In fact, a finitely generated pair $(G, N)$ is relatively simple if and only if $N$ is the \emph{unique} normal subgroup of $G$ such that $G/N$ is simple (Lemma~\ref{lem:unique}).

Our first main result is that abstract twisted Brin--Thompson groups are always relatively simple.

\begin{lettertheorem}[Theorem~\ref{thrm:rel_simple}]
\label{thrm:main:rel_simple}
Every abstract twisted Brin--Thompson group is relatively simple. More precisely, every proper normal subgroup of the abstract twisted Brin--Thompson group $SV_G$ lies in the kernel of the map $SV_G\to SV_{G/\ker(G\curvearrowright S)}$.
\end{lettertheorem}

The \emph{canonical kernel} in $SV_G$ is defined to be
\[
SK_G \coloneqq \ker(SV_G \to SV_{G/\ker(G\curvearrowright S)}) \text{,}
\]
where $(G/\ker(G\curvearrowright S))\curvearrowright S$ is the action induced by $G\curvearrowright S$. Thus, Theorem~\ref{thrm:main:rel_simple} says that the normal pair $(SV_G,SK_G)$ is relatively simple.

\medskip

\subsection{Finiteness properties} A group is of \emph{type~$\F_n$} if it has a classifying space with finite $n$-skeleton, and of \emph{type~$\F_\infty$} if it is of type $\F_n$ for all $n$. All groups are of type $\F_0$, type $\F_1$ is equivalent to finite generation, and type $\F_2$ is equivalent to finite presentability. There has been much interest in studying finiteness properties of Thompson-like groups over the years \cite{browngeo84,brown87,bux16, skipper19, LISW25}, and here we inspect finiteness properties of abstract twisted Brin--Thompson groups. First let us define the following family of properties of group actions.

\begin{definition}
Let $G$ be a group acting on a non-empty set $S$. For $n\in\N$, we say the action is of \emph{type~$[\A_n]$} if the following hold:
\begin{enumerate}
    \item $G$ is of type~$\F_n$.
    \item For any finite $T\subseteq S$, $\Stab_G(T)$ is of type~$\F_{n-|T|}$.
    \item The diagonal action of $G$ on $S^n$ has finitely many orbits.
\end{enumerate}
The action is of \emph{type~$[\A_\infty]$} if it is of type~$[\A_n]$ for all $n\in\N$. If an action of type~$[\A_n]$ (with $n\in\N\cup\{\infty\}$) is moreover faithful, we say that it is of \emph{type~$(\A_n)$}.
\end{definition}

The use of brackets versus parentheses is inspired by the notation of Bader--Sauer for higher Kazhdan properties \cite{higherT}. Note that the condition in the second item is trivially satisfied for any $T \subseteq S$ with $|T| \geq n$. Actions of type~$[\A_n]$ have an established tight connection to finiteness properties, thanks to permutational wreath products. Indeed, for an infinite group $B$ of type $\F_n$, the permutational wreath product $B \wr_S G$ is of type $\F_n$ if and only if the action $G\curvearrowright S$ is of type~$[\A_n]$ \cite{cornulier06,bartholdi15,FFKLZ}.

It was shown in \cite{zaremsky_fp_tbt} that for $G$ acting faithfully on $S$, the twisted Brin--Thompson group $SV_G$ is finitely presented if and only if the (faithful) action $G\curvearrowright S$ is of type~$(\A_2)$. (This was mostly called type~$(\A)$ in \cite{zaremsky_fp_tbt} and in subsequent work, where the focus was only on finite presentability, though ``type~$(\A_n)$'' was also introduced in \cite[Remark~2.8]{zaremsky_fp_tbt}.) Moreover, $SV_G$ is finitely generated if and only if $G\curvearrowright S$ is of type~$(\A_1)$, and type~$\F_\infty$ if and only if $G\curvearrowright S$ is of type~$(\A_\infty)$ \cite{belk22,zaremsky_fp_tbt}. We prove that these same results hold in the non-faithful case:

\begin{lettertheorem}[Theorem~\ref{thrm:fin_props}]
\label{thrm:main:fp}
Let $G$ be a group acting on a non-empty set $S$. Then $SV_G$ is finitely generated if and only if $G\curvearrowright S$ is of type~$[\A_1]$, finitely presented if and only if $G\curvearrowright S$ is of type~$[\A_2]$, and type~$\F_\infty$ if and only if $G\curvearrowright S$ is of type~$[\A_\infty]$.
\end{lettertheorem}

An interesting consequence is that if $SV_G$ is finitely presented, then its simple quotient $SV_{G/\ker(G\curvearrowright S)}$ is finitely presented if and only if $G/\ker(G\curvearrowright S)$ is (Corollary~\ref{cor:push_fin_props}).

As for the higher finiteness properties, we conjecture that $SV_G$ is of type~$\F_n$ if and only if $G\curvearrowright S$ is of type~$[\A_n]$. This is even open for $G\curvearrowright S$ faithful \cite[Conjecture~H]{belk22} (outside the $n=1,2,\infty$ cases). In the faithful case, \cite[Section~4]{zaremsky_fp_tbt} makes clear that one can achieve $SV_G$ being of type~$\F_n$ with stronger hypotheses than $G\curvearrowright S$ being of type~$(\A_n)$. For example, for each $n$ there exists $k\ge n$ such that if $G$ and $\Stab_G(S')$ are of type~$\F_n$ for all $S'\subseteq S$ with $|S'|\le k$ and the diagonal action of $G$ on $S^k$ has finitely many orbits, then $SV_G$ is of type~$\F_n$. This is likely true in the non-faithful case as well, but we will not pursue this here, since it would be better to be able to weaken the hypotheses to $G\curvearrowright S$ being of type~$[\A_n]$.

\subsection{Embeddings} Recall the well-known \emph{Boone--Higman conjecture} first posed in \cite{boone74}, and the related \emph{permutational Boone--Higman conjecture}, posed for example in \cite{zaremsky_fp_tbt}:

\begin{conjecture}[Boone--Higman (BH)]\label{conj:bh}
Every finitely generated group with solvable word problem embeds in a finitely presented simple group.
\end{conjecture}

\begin{conjecture}[Permutational Boone--Higman (PBH)]\label{conj:pbh}
Every finitely generated group with solvable word problem embeds in a group admitting an action of type~$(\A_2)$ (equivalently embeds in a finitely presented simple twisted Brin--Thompson group).
\end{conjecture}

Here the ``equivalently'' in the parenthetical is thanks to \cite{zaremsky_fp_tbt,BFFHZ}.

One of the main interests in the Boone--Higman conjecture is that it would provide an elegant algebraic characterization for solvability of the word problem: indeed, an easy algorithm that goes back to Kuznetsov \cite{kuznetsov58} shows that subgroups of finitely presented simple groups have solvable word problem. On the other hand, the permutational Boone--Higman conjecture proposes to replace the algebraic property of being simple with the more geometric property of the existence of a certain permutation representation, and also highlights the outsized importance of Thompson-like groups among finitely presented simple groups. Note that (PBH)$\Rightarrow$(BH) for any given group. In addition to the question of whether either of these conjectures is true, it is also a major question whether (PBH)$\Leftrightarrow$(BH), or equivalently whether twisted Brin--Thompson groups are universal among finitely presented simple groups. By now the Boone--Higman conjecture is known for many families of groups, e.g., $\Q$-linear groups \cite{scott84, zaremsky_fpss}, $\Aut(F_n)$ \cite{BFFHZ}, hyperbolic groups \cite{bbmz_hyp}, Baumslag--Solitar and free-by-cyclic groups \cite{bux_boonehigman}, and more \cite{bbmz_survey}; in all these cases, the groups not only satisfy (BH) but even (PBH).

\medskip

Here we formulate ``relative'' versions of these conjectures and prove a surprising connection with the classical conjectures (Theorem~\ref{thrm:main:embed}).

\begin{definition}[Embeds, sharply embeds]
Given two normal pairs $(G,N)$ and $(G',N')$, we say $(G,N)$ \emph{embeds} in $(G',N')$ if there exists an injective homomorphism $\iota\colon G\hookrightarrow G'$ such that $\iota(G)\cap N'=\iota(N)$. We say a group $\Gamma$ \emph{sharply embeds} in a normal pair $(G,N)$ if the normal pair $(\Gamma,\{1\})$ embeds in $(G,N)$. We say a group $\Gamma$ \emph{sharply embeds} in a relatively simple group $G$ if $\Gamma$ sharply embeds in the normal pair $(G,N)$ for $N$ the largest normal subgroup of $G$.
\end{definition}

\begin{conjecture}[Relative Boone--Higman (relBH)]\label{conj:relbh}
Let $\Gamma$ be a finitely generated group with solvable word problem. Then $\Gamma$ sharply embeds in a finitely presented relatively simple group.
\end{conjecture}

\begin{conjecture}[Relative permutational Boone--Higman (relPBH)]\label{conj:relpbh}
Let $\Gamma$ be a finitely generated group with solvable word problem. Then there exists a group $G$ with a type~$[\A_2]$ action on a set $S$ such that $\Gamma$ sharply embeds in $(G,\ker(G\curvearrowright S))$ (and hence $\Gamma$ sharply embeds in the finitely presented relatively simple group $SV_G$).
\end{conjecture}

Here the ``hence'' in the parenthetical follows from Theorems~\ref{thrm:main:rel_simple} and~\ref{thrm:main:fp} and Lemma~\ref{lem:embed_in_tbt}. It would be interesting to know whether the converse of this parenthetical is true, i.e., whether every finitely presented abstract twisted Brin--Thompson group sharply embeds in a normal pair of the form $(G,\ker(G\curvearrowright S))$ for some $G\curvearrowright S$ of type~$[\A_2]$ (see Question~\ref{quest:tBT_to_action}).

As in the case of the classical Boone--Higman conjecture, these conjectures would also provide an elegant characterization of solvability of the word problem, thanks to essentially the same algorithm; in fact the relative Boone--Higman conjecture, the weakest of the four, is a natural strengthening of a characterization of groups with solvable word problem that follows from the Boone--Higman--Thompson theorem, and that we present in Appendix~\ref{appendix} (see Proposition \ref{prop:strong_bht}). In particular the converses of these conjectures are true. We should emphasize that the desired embeddings in these relative conjectures need to be sharp. Indeed, every finitely presented group $G$ (regardless of the word problem) embeds non-sharply in a group admitting a type~$[\A_2]$ action, namely $G$ acting trivially on a point.

Clearly we have the following implications for a given group:
\begin{figure}[H]
\centering
\begin{tikzpicture}
\node at (0,0) {(PBH)};
\node at (1,0) {$\Rightarrow$};
\node at (2,0) {(BH)};
\node at (0,-0.5) {$\Downarrow$};
\node at (0,-1) {(relPBH)};
\node at (1,-1) {$\Rightarrow$};
\node at (2,-1) {(relBH)};
\node at (2,-0.5) {$\Downarrow$};
\end{tikzpicture}
\end{figure}

Our main result in this vein is that in fact (BH)$\Rightarrow$(relPBH).

\begin{lettertheorem}[Theorem~\ref{thrm:embedding}]
\label{thrm:main:embed}
If a group satisfies the Boone--Higman conjecture then it satisfies the relative permutational Boone--Higman conjecture.
\end{lettertheorem}

We emphasize that the relative permutational Boone--Higman conjecture is a statement entirely about embedding in a group admitting a certain action, with no \emph{a priori} reference to (relative) simplicity or Thompson-like groups. Theorem~\ref{thrm:main:embed} and Lemma~\ref{lem:embed_in_tbt} imply that if a given group satisfies the Boone--Higman conjecture, meaning it embeds in a finitely presented simple group that a priori has nothing to do with Thompson groups, then in fact it sharply embeds in a finitely presented relatively simple group that is definitively from the Thompson world, namely an abstract twisted Brin--Thompson group.

Since it is a major question whether (BH)$\Rightarrow$(PBH), an obvious followup question to Theorem~\ref{thrm:main:embed} is whether (relPBH)$\Rightarrow$(PBH), and relatedly whether (relBH)$\Rightarrow$(BH). See Question~\ref{quest:relBH_to_BH} for a more precise formulation.

\subsection{Other properties} Finally, we investigate a number of other properties of abstract twisted Brin--Thompson groups. These have all already been studied to some extent in the special cases of faithful twisted Brin--Thompson groups and/or labeled Thompson groups, though some of our results here are new even in those cases.

First, we prove that abstract twisted Brin--Thompson groups are always uniformly perfect (Theorem~\ref{thrm:up}) and have property NL (Theorem~\ref{thrm:NL}), meaning that they cannot act on a hyperbolic space with a loxodromic element. Uniform perfectness was previously known for faithful twisted Brin--Thompson groups \cite{NL} and labeled Thompson groups \cite{wu25}. Property NL was known in the faithful case \cite{NL}, but to the best of our knowledge is even new for labeled Thompson groups. Thanks to a criterion of Genevois \cite{NL:V}, we deduce that abstract twisted Brin--Thompson groups have property FW$_\infty$, meaning that every action on a finite-dimensional CAT(0) cube complex fixes a point.

Next we look at homological properties, and prove that all abstract twisted Brin--Thompson groups are boundedly acyclic (Theorem~\ref{thrm:SV_G_bdd_acyc}) and $\ell^2$-invisible (Corollary~\ref{cor:tBT_l2_invis}). Bounded acyclicity was already known in the faithful case and for labeled Thompson groups \cite{wu25}, while $\ell^2$-invisibility was only previously known for labeled Thompson groups with non-amenable label group \cite{wu25} and for (non-twisted) Brin--Thompson groups \cite{sauer:thumann, thumann:l2}. The key is a new criterion for $\ell^2$-invisibility (Theorem~\ref{thrm:general_l2_criterion}) that is interesting in its own right. Let us state it in the special case of homeomorphism groups, in which case it parallels a criterion for $C^*$-simplicity by Le Boudec--Matte Bon \cite{leboudec:simplicity:homeo}.

\begin{lettertheorem}[Corollary~\ref{cor:l2:faithful}]
\label{thrm:criterion:intro}
Let $\Gamma$ be a group acting faithfully by homeomorphisms on a Hausdorff space $X$. Suppose that for every non-empty open set $U$, the rigid stabilizer $\RStab_{\Gamma}(U)$ is non-amenable. Then $\Gamma$ is $\ell^2$-invisible.
\end{lettertheorem}

This can be used to show that Thompson's group $F$ is non-amenable if and only if it is $\ell^2$-invisible (Corollary~\ref{cor:F_l2}).

Finally, we characterize precisely when an abstract twisted Brin--Thompson group $SV_G$ is $C^*$-simple. The answer (Theorem~\ref{thrm:tbt_C*_simple}) is that it happens if and only if $SV_G$ has trivial amenable radical, which is equivalent to just requiring $\ker(G\curvearrowright S)$ to have no amenable normal subgroups. It is curious that neither $G$ nor $\ker(G\curvearrowright S)$ need to be $C^*$-simple, for $SV_G$ to be $C^*$-simple.

\subsection*{Outline}

This paper is organized as follows. In Section~\ref{sec:groups} we define abstract twisted Brin--Thompson groups. In Section~\ref{sec:simple} we prove Theorem~\ref{thrm:main:rel_simple}, that they are relatively simple. In Section~\ref{sec:fin_props} we prove Theorem~\ref{thrm:main:fp} regarding finiteness properties. In Section~\ref{sec:embed} we investigate embedding properties and prove Theorem~\ref{thrm:main:embed}. In Section~\ref{sec:geom} we investigate geometric properties, including actions on hyperbolic spaces. In Section~\ref{sec:homological} we look at homological properties, and prove that every abstract twisted Brin--Thompson group is boundedly acyclic and $\ell^2$-invisible. Finally, in Section~\ref{sec:Cstar} we characterize precisely which abstract twisted Brin--Thompson groups are $C^*$-simple. In Appendix~\ref{appendix} we prove a strengthening of the Boone--Higman--Thompson theorem (Proposition~\ref{prop:strong_bht}) that has an interesting connection to the relative Boone--Higman conjecture.

\subsection*{Acknowledgments} This work has benefited from discussions with Carolyn Abbott, Jim Belk, Collin Bleak, Matt Brin, Kai-Uwe Bux, James Hyde, Adrien Le Boudec, Robbie Lyman, Francesco Matucci, Roman Sauer, and Jakub Tucker. FFF is supported by the Herchel Smith Postdoctoral Fellowship Fund.

\section{Construction of the groups}\label{sec:groups}

In this section we construct our abstract twisted Brin--Thompson groups. This will proceed in stages.

\subsection{Brin--Thompson groups}\label{ssec:bt}

First let us recall the construction of the (non-twisted) Brin--Thompson group $SV$, for $S$ a non-empty set. We will use the setup and notation from \cite{belk22}, where the twisted case was introduced; see also \cite{zar_taste} for a shorter introduction. Let $C=\{0,1\}^\N$ be the Cantor set and $C^S$ the space of functions from $S$ to $C$ with the usual product topology. Let $\{0,1\}^*$ be the set of finite words in $\{0,1\}$. Given a function $\psi\colon S\to \{0,1\}^*$ such that $\psi(s)=\varnothing$ (the empty word) for all but finitely many $s\in S$, denote by $B(\psi)$ the basic open set
\[
B(\psi) \coloneqq \{\kappa\in C^S\mid \psi(s)\text{ is a prefix of }\kappa(s)\text{ for all } s\in S\}\text{.}
\]
Call any such $B(\psi)$ a \emph{dyadic brick}. For any such $\psi$, the \emph{canonical homeomorphism}
\[
h_\psi \colon C^S \to B(\psi)
\]
is defined by
\[
h_\psi(\kappa)(s) \coloneqq \psi(s)\cdot\kappa(s)\text{,}
\]
where $\cdot$ denotes concatenation. Given two such $\psi$ and $\varphi$, we get the \emph{canonical homeomorphism}
\[
h_{\psi\to\varphi} \coloneqq h_\varphi \circ h_\psi^{-1} \colon B(\psi) \to B(\varphi).
\]

\begin{definition}[Brin--Thompson group]
The \emph{Brin--Thompson group} $SV$ is the group of all homeomorphisms $h$ of $C^S$ constructed as follows. Partition the domain $C^S$ into finitely many dyadic bricks $B(\psi_1),\dots,B(\psi_n)$, partition the range $C^S$ into the same number of dyadic bricks $B(\varphi_1),\dots,B(\varphi_n)$, and define $h$ to map each $B(\psi_i)$ to $B(\varphi_i)$ via the canonical homeomorphism $h_{\psi_i\to\varphi_i}$.
\end{definition}

Intuitively, $SV$ is the group of all ``piecewise prefix replacements'' of $C^S$.

\subsection{Brin--Thompson groupoids}\label{ssec:bt_oid}

It is often convenient to work with the Brin--Thompson groupoid $S\mathcal{V}$, which we recall now. For $m\in\N$ let $C^S(m)$ be the disjoint union of $m$ copies of $C^S$.

\begin{definition}[Brin--Thompson groupoid, rank, corank]
The \emph{Brin--Thompson groupoid} $S\mathcal{V}$ is the groupoid of all homeomorphisms from $C^S(m)$ to $C^S(n)$ ($m,n\in\N$) obtained by partitioning the domain and range into the same number of dyadic bricks (meaning dyadic bricks within each copy of $C^S$) and mapping the domain ones to the range ones via canonical homeomorphisms. An element with this domain and range has \emph{rank} $n$ and \emph{corank} $m$.
\end{definition}

In particular, $SV$ is the subgroup of $S\mathcal{V}$ consisting of elements with rank and corank~1. Let us recall some basic elements and operations on $S\mathcal{V}$.

\begin{definition}[Simple split]
For $s\in S$, the \emph{simple split} $x_s$ is the element of $S\mathcal{V}$ with corank 1 and rank 2 given by partitioning $C^S$ into $B(\psi_0^s)$ and $B(\psi_1^s)$, where $\psi_i^s$ ($i=0,1$) sends $s$ to $i$ and $s'\ne s$ to $\varnothing$, and mapping $C^S$ to $C^S(2)$ by sending $B(\psi_0^s)$ to the first copy of $C^S$ and $B(\psi_1^s)$ to the second copy, both via canonical homeomorphisms.
\end{definition}

\begin{definition}[Permutation]
For $m\in\N$ and $\sigma\in \Sigma_m$, the \emph{permutation} $p_\sigma\in S\mathcal{V}$ is the element with rank and corank $m$ that permutes the copies of $C^S$ in $C^S(m)$ via $\sigma$. Write $\mathcal{S}(m)$ for the copy of $\Sigma_m$ in $S\mathcal{V}$ consisting of all the $p_\sigma$ for $\sigma\in \Sigma_m$.
\end{definition}

\begin{definition}[Direct sum]
Given $h,h'\in S\mathcal{V}$, say $h$ has corank $m$ and rank $n$ and $h'$ has corank $m'$ and rank $n'$, the \emph{direct sum} $h\oplus h'$ is the element of $S\mathcal{V}$ with corank $m+m'$ and rank $n+n'$ that sends the first $m$ copies of $C^S$ in the domain to the first $n$ copies of $C^S$ in the range via $h$, and sends the last $m'$ copies of $C^S$ in the domain to the last $n'$ copies of $C^S$ in the range via $h'$.
\end{definition}

Certain partitions into dyadic bricks are of particular interest, namely those that are encoded into the leaves of a multicolored tree. Let us explain all this.

\begin{definition}[Multicolored forest/tree]
Consider the subset of $S\mathcal{V}$ that contains all simple splits and is closed under compositions and direct sums. The elements of this subset are called \emph{multicolored forests}, and the elements of corank~1 are called \emph{multicolored trees}. When we want to encode the dependence on $S$ we may say \emph{$S$-multicolored}. The \emph{roots} of a multicolored forest are the copies of $C^S$ in its domain, and its \emph{leaves} are the copies of $C^S$ in its range.
\end{definition}

This definition is slightly different than the one in \cite{belk22}. There, permutations were also allowed. For our purposes it is more convenient to define multicolored forests this way.
We think of multicolored forests pictorially as in Figure~\ref{fig:MF}.

\begin{figure}[htb]
\centering
\begin{tikzpicture}[line width=1pt]

\draw[color=red] (0,0) -- (1,1) -- (0,2);
\draw[color=blue] (-1,-0.5) -- (0,0) -- (-1,0.5)   (-1,1.5) -- (0,2) -- (-1,2.5);
\draw[color=red] (-2,1) -- (-1,1.5) -- (-2,2);
\draw[color=ForestGreen] (-2,-1) -- (-1,-0.5) -- (-2,0);

\end{tikzpicture}
\caption{The multicolored tree $(x_g\oplus 1 \oplus x_r\oplus 1)\circ (x_b\oplus x_b)\circ x_r \colon C^S \to C^S(6)$, where $S=\{r,b,g\}$, with the simple split $x_r$ colored red, $x_b$ blue, and $x_g$ green.
It defines an arboreal partition whose blocks $B(\psi_1), \ldots, B(\psi_6)$, ordered bottom to top, satisfy e.g., $\psi_4(r) = 10, \psi_4(b) = 0, \psi_4(g) = \varnothing$.}
\label{fig:MF}
\end{figure}
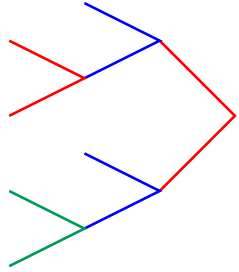

\begin{definition}[Arboreal partition]
Let $m,n\in\N$. A partition of $C^S(m)$ into $n$ dyadic bricks is called \emph{arboreal} if there exists a multicolored forest $F\colon C^S(m) \to C^S(n)$ such that the blocks of the partition are the preimages under $F$ of the $n$ copies of $C^S$ in the range $C^S(n)$.
\end{definition}

As elements of the groupoid $S\mathcal{V}$, multicolored forests satisfy various interesting relations, among themselves and also in conjunction with permutations. Writing $1_m$ for the identity on $C^S(m)$, we quickly see for example that if $F$ has corank $m$ and rank $n$ and $F'$ has corank $m'$ and rank $n'$ then
\[
(F\oplus 1_{n'})\circ (1_m \oplus F') = (1_n \oplus F')\circ (F\oplus 1_{m'}) = F\oplus F'\text{.}
\]
Also, if $T_1,\dots,T_n$ are multicolored trees with a total of $m$ leaves, and $\sigma\in \Sigma_n$, then
\[
(T_1\oplus\cdots\oplus T_n)\circ p_\sigma = p_{\sigma'}\circ (T_{\sigma(1)}\oplus\cdots\oplus T_{\sigma(n)})
\]
for some $\sigma'\in \Sigma_m$. A less obvious relation, called a \emph{cross relation}, is that for any $s\ne t$ in $S$ we have
\[
(x_t \oplus x_t)\circ x_s = p_{(2~3)}\circ (x_s \oplus x_s)\circ x_t\text{,}
\]
as depicted in Figure~\ref{fig:cross}. One can check that these really do define the same homeomorphism.
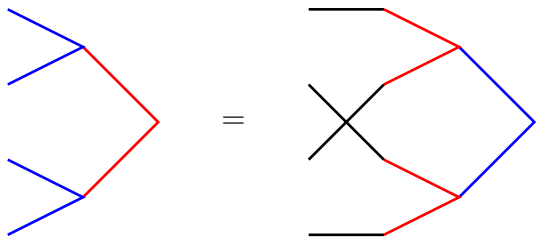
\begin{figure}[htb]
\centering
\begin{tikzpicture}[line width=1pt]

\draw[color=red] (0,0) -- (1,1) -- (0,2);
\draw[color=blue] (-1,-0.5) -- (0,0) -- (-1,0.5)   (-1,1.5) -- (0,2) -- (-1,2.5);
\node at (2,1) {$=$};

\begin{scope}[xshift=5cm]
\draw[color=blue] (0,0) -- (1,1) -- (0,2);
\draw[color=red] (-1,-0.5) -- (0,0) -- (-1,0.5)   (-1,1.5) -- (0,2) -- (-1,2.5);
\draw (-1,0.5) -- (-2,1.5)   (-1,1.5) -- (-2,0.5)   (-1,-0.5) -- (-2,-0.5)   (-1,2.5) -- (-2,2.5);
\end{scope}

\end{tikzpicture}
\caption{The cross relation $(x_b \oplus x_b)\circ x_r = p_{(2~3)}\circ (x_r \oplus x_r)\circ x_b$, with $x_r$ colored red and $x_b$ colored blue, and $p_{(2~3)} \in \mathcal{S}(4)$.}
\label{fig:cross}
\end{figure}

For every element of $S\mathcal{V}$, the partitions of the domain and range into dyadic bricks used to define the element can be chosen to be arboreal. This is explained for example in \cite[Remark~1.1]{belk22}, where arboreal partitions are called ``dyadic partitions''. Hence, every element of $S\mathcal{V}$ is of the form $F_-^{-1} \circ p_\sigma \circ F_+$ for some multicolored forests $F_-$ and $F_+$ with the same number of leaves, say $n$, and $p_\sigma\in \mathcal{S}(n)$.

\begin{definition}
Denote the element $F_-^{-1} \circ p_\sigma \circ F_+$ by
\[
[F_-,\sigma,F_+].
\]
Call the triple $(F_-,\sigma,F_+)$ a \emph{representative triple} of $[F_-,\sigma,F_+]$.
\end{definition}

An element of $S\mathcal{V}$ can have many different representative triples. For example, thanks to the cross relation, we have that $[(x_t\oplus x_t)\circ x_s,(2~3),(x_s\oplus x_s)\circ x_t]=[1_1,\id,1_1]$ for any $s,t\in S$. Less immediately, there is the notion of ``expansion'', which we recall now.

For a multicolored forest $F$ with $n$ leaves and $s\in S$ write
\[
F_k^s\coloneqq (1_{k-1}\oplus x_s \oplus 1_{n-k})\circ F\text{.}
\]
For $1\le k\le n$ and $\sigma\in \Sigma_n$, viewing $\sigma$ as a bijection from $\{1,\dots,n\}$ to itself, let $(\sigma)\varsigma_k^n$ be the bijection
\begin{align*}
(\sigma)\varsigma_k^n \colon &\left\{1,\dots,k-1,k-\frac{1}{2},k+\frac{1}{2},k+1,\dots,n\right\} \\
&\to \left\{1,\dots,\sigma(k)-1,\sigma(k)-\frac{1}{2},\sigma(k)+\frac{1}{2},\sigma(k)+1,\dots,n\right\}
\end{align*}
sending each $i\ne k$ to $\sigma(i)$ and sending $k-\frac{1}{2}$ and $k+\frac{1}{2}$ to $\sigma(k)-\frac{1}{2}$ and $\sigma(k)+\frac{1}{2}$ respectively. Identifying both the domain and codomain of $(\sigma)\varsigma_k^n$ with $\{1,\dots,n+1\}$ in the unique order-preserving way, we get that $(\sigma)\varsigma_k^n\in \Sigma_{n+1}$. Our convention of writing $\varsigma_k^n$ to the right of its argument comes from \cite{witzel18}, and is convenient for keeping track of other left versus right choices.

\begin{definition}[Expansion]
Let $(F_-,\sigma,F_+)$ represent an element of $S\mathcal{V}$, say $F_-$ and $F_+$ have $n$ leaves. The triple $((F_-)_{\sigma(k)}^s,(\sigma)\varsigma_k^n,(F_+)_k^s)$ is called an \emph{expansion} of $(F_-,\sigma,F_+)$, specifically the $k$th expansion with color $s$.
\end{definition}

Expansions do not change the element of $S\mathcal{V}$ being represented:

\begin{observation}\label{obs:expand}
For any expansion $((F_-)_{\sigma(k)}^s,(\sigma)\varsigma_k^n,(F_+)_k^s)$ of $(F_-,\sigma,F_+)$, we have
\[
[(F_-)_{\sigma(k)}^s,(\sigma)\varsigma_k^n,(F_+)_k^s] = [F_-,\sigma,F_+]\text{.}
\]
\end{observation}

\begin{proof}
For any dyadic bricks $B(\psi)$ and $B(\varphi)$ and any $i \in \{0, 1\}$, the restriction of the canonical homeomorphism $h_{\psi\to\varphi}$ to $B(\psi_i^s)$ has range $B(\varphi_i^s)$ and is the canonical homeomorphism $h_{\psi_i^s\to\varphi_i^s}$. Thus the result is immediate from the definition of expansion.
\end{proof}

\begin{lemma}\label{lem:upper_bound}
Given two multicolored forests $F,F'$ with the same corank, there exist multicolored forests $E,E'$ and a permutation $\sigma$ such that $E\circ F = p_\sigma \circ E'\circ F'$.
\end{lemma}

\begin{proof}
It is enough to prove the result for $F$ and $F'$ multicolored trees, so let us call them $T$ and $T'$. We induct on the rank $n$ of $T$. The base case is $n=1$, i.e., $T=\cdot$, where the result holds trivially using $E=T'$ (and $E'$ and $\sigma$ trivial). Now suppose $n\ge 2$, so we can write $T = (L\oplus R) \circ x_s$ for some $s\in S$ and some multicolored trees $L,R$ with ranks summing to $n$. Let $E'=x_s\oplus\cdots\oplus x_s$, with a number of terms so that $E' \circ T'$ makes sense, and note that thanks to cross relations we can write $E'\circ T' = p_\sigma\circ (L'\oplus R')\circ x_s$ for some $\sigma$ and some multicolored trees $L'$ and $R'$. Now in order to prove the desired result for $T$ and $T'$, it is sufficient to prove it for $T$ and $E'\circ T'$, and for this it is sufficient to prove it for $L$ and $L'$, and for $R$ and $R'$. But each of $L$ and $R$ has rank strictly smaller than $n$, so indeed these hold by induction.
\end{proof}

Now the multiplication in $S\mathcal{V}$ can be described using representative triples. Consider two elements $h,h'\in S\mathcal{V}$, say $h=[F_-,\sigma,F_+]$ and $h'=[F_-',\sigma',F_+']$, such that the product $h\circ h'$ exists. Then $F_+$ and $F_-'$ have the same corank, so by Lemma~\ref{lem:upper_bound} we can perform expansions until without loss of generality $F_+=F_-'$ (and by Observation~\ref{obs:expand} we have not changed $h$ or $h'$). Now the product is
\[
h\circ h' = [F_-,\sigma\sigma',F_+']\text{.}
\]

\subsection{Abstract twisted Brin--Thompson group(oid)s}\label{ssec:abstract}

Now we can define our new family of groups. Let $G$ be a group acting from the left on a non-empty set $S$. Start with the groupoid $S\mathcal{V}$, so the set of objects is indexed by $\N$, namely the set of $C^S(n)$. Note that $S\mathcal{V}$ is generated (as a groupoid) by all the multicolored forests together with the permutation groups $\mathcal{S}(n)$. For each $n\in\N$ let $\mathcal{G}(n)$ be a copy of the group $G^n$ based at the object $n$. Let $S\mathcal{V}_G$ be the groupoid with set of objects $\N$ generated by $S\mathcal{V}$ and $\mathcal{G}(n)$ for each $n\in\N$, subject to the set of relations that is closed under compositions and direct sums and contains the following:
\begin{itemize}
    \item $x_{g.s} \circ g = (g\oplus g)\circ x_s$ for all $s\in S$ and $g\in G$.
    \item $(g_1\oplus\cdots\oplus g_n)\circ p_\sigma = p_\sigma \circ (g_{\sigma(1)}\oplus\cdots\oplus g_{\sigma(n)})$ for all $\sigma\in \Sigma_n$ and $g_1,\dots,g_n\in G$.
\end{itemize}
See Figure~\ref{fig:relations} for examples of these relations.
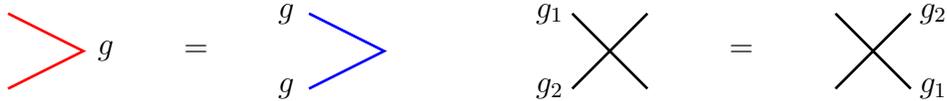
\begin{figure}[htb]
\centering
\begin{tikzpicture}[line width=1pt]
\draw[color=red] (0,0.5) -- (1,1) -- (0,1.5);
\node at (1.3,1) {$g$};
\node at (2.5,1) {$=$};
\draw[color=blue] (4,0.5) -- (5,1) -- (4,1.5);
\node at (3.7,0.5) {$g$};
\node at (3.7,1.5) {$g$};

\begin{scope}[xshift=7.5cm]
\draw (0,0.5) -- (1,1.5)   (0,1.5) -- (1,0.5);
\node at (-0.3,0.5) {$g_2$};
\node at (-0.3,1.5) {$g_1$};
\node at (2.25,1) {$=$};

\draw (3.5,0.5) -- (4.5,1.5)   (3.5,1.5) -- (4.5,0.5);
\node at (4.8,0.5) {$g_1$};
\node at (4.8,1.5) {$g_2$};
\end{scope}
\end{tikzpicture}
\caption{The relations $x_r\circ g = (g\oplus g)\circ x_b$ and $(g_1\oplus g_2)\circ p_{(1~2)} = p_{(1~2)}\circ (g_2\oplus g_1)$, with $x_r$ colored red, $x_b$ colored blue, and $g$ such that $g.b = r$.}
\label{fig:relations}
\end{figure}

Call $S\mathcal{V}_G$ the \emph{abstract twisted Brin--Thompson groupoid} for $G\curvearrowright S$.

Intuitively, this first relation says that if we ``do'' $g$ and then split in half in the $g.s$-direction, this is the same as first splitting in half in the $s$-direction and then ``doing'' $g$ to each half. Note that since we do not require $G\curvearrowright S$ to be faithful, it is possible that $x_{g.s}=x_s$ for all $s\in S$ even for $g\ne 1$. The second relation ensures that for each $n$ the groups $\mathcal{S}(n)$ and $\mathcal{G}(n)$ generate a copy of the wreath product $\Sigma_n \ltimes G^n$, which we denote by
\[
\mathcal{W}(n) \coloneqq \mathcal{S}(n) \ltimes \mathcal{G}(n)\text{.}
\]
(Since we are taking the action of $\Sigma_n$ on $\{1,\dots,n\}$ to be a left action, the action of $\Sigma_n$ on $G^n$ is a right action, hence the order of the factors here.)

\medskip

Thanks to the above relations, every element of $S\mathcal{V}_G$ is of the form $F_-^{-1} \circ p_\sigma \circ (g_1,\dots,g_n) \circ F_+$, for $F_-$ and $F_+$ multicolored forests with the same number of leaves, say $n$, $p_\sigma\in \mathcal{S}(n)$, and $(g_1,\dots,g_n)\in \mathcal{G}(n)$. Denote this element by
\[
[F_-,\sigma,(g_1,\dots,g_n),F_+]\text{,}
\]
and call $(F_-,\sigma,(g_1,\dots,g_n),F_+)$ a \emph{representative quadruple} for this element. The \emph{rank} of this element is the number of roots of $F_+$ and the \emph{corank} is the number of roots of $F_-$. We have a completely analogous notion of expansion as in the $G=\{1\}$ case, which also does not change the element; the $k$th expansion with color $s\in S$ of $(F_-,\sigma,(g_1,\dots,g_n),F_+)$ is
\[
((F_-)_{\sigma(k)}^{(s)g_k},(\sigma)\varsigma_k^n,(g_1,\dots,g_{k-1},g_k,g_k,g_{k+1},\dots,g_n),(F_+)_k^s)\text{.}
\]
The groupoid multiplication can be computed using representative quadruples by taking two elements $[F_-,\sigma,(g_1,\dots,g_n),F_+]$ and $[E_-,\tau,(h_1,\dots,h_m),E_+]$, doing expansions until without loss of generality $F_+=E_-$ (and so $m=n$), and then computing
\[
[F_-,\sigma,(g_1,\dots,g_n),F_+][E_-,\tau,(h_1,\dots,h_n),E_+]\coloneqq [F_-,\sigma\tau,(g_{\tau(1)}h_1,\dots,g_{\tau(n)}h_n),E_+]\text{.}
\]

\begin{definition}[Abstract twisted Brin--Thompson group, canonical kernel]
\label{def:main}

Let $G$ be a group acting on a non-empty set $S$. The \emph{abstract twisted Brin--Thompson group} $SV_G$ is the subgroup of $S\mathcal{V}_G$ consisting of elements with rank and corank~1. Thus elements of $SV_G$ are represented by quadruples
$(T_-,\sigma,(g_1,\dots,g_n),T_+)$, where $T_-$ and $T_+$ are multicolored trees. The \emph{canonical kernel} $SK_G$ is the subgroup of elements of the form $[T,\id,(k_1,\dots,k_n),T]$ for $T$ a multicolored tree and $k_1, \ldots, k_n \in \ker(G \curvearrowright S)$.
\end{definition}

See Figure~\ref{fig:elements} for an example of an element of $SV_G$ and an element of $SK_G$.

\begin{figure}[htb]
\centering
\begin{tikzpicture}[line width=1pt]
\draw[color=red] (1,-1) -- (0,0) -- (1,1);
\draw[color=blue] (2,-1.5) -- (1,-1) -- (2,-0.5);
\draw[color=ForestGreen] (2,0.5) -- (1,1) -- (2,1.5);
\draw (2,-1.5) -- (3,-0.5)   (2,-0.5) -- (3,1.5)   (2,0.5) -- (3,-1.5)   (2,1.5) -- (3,0.5);
\node at (3.25,-1.5) {$g_1$};
\node at (3.25,-0.5) {$g_2$};
\node at (3.25,0.5) {$g_3$};
\node at (3.25,1.5) {$g_4$};
\draw[color=red] (3.5,-1.5) -- (4,-1) -- (3.5,-0.5);
\draw[color=blue] (4,-1) -- (4.5,-0.5) -- (3.5,0.5);
\draw[color=red] (4.5,-0.5) -- (5,0) -- (3.5,1.5);

\begin{scope}[xshift=7cm]
\draw[color=red] (1,-1) -- (0,0) -- (1,1);
\draw[color=blue] (2,-1.5) -- (1,-1) -- (2,-0.5);
\draw[color=ForestGreen] (2,0.5) -- (1,1) -- (2,1.5);
\node at (2.25,-1.5) {$k_1$};
\node at (2.25,-0.5) {$k_2$};
\node at (2.25,0.5) {$k_3$};
\node at (2.25,1.5) {$k_4$};
\draw[color=red] (3.5,-1) -- (4.5,0) -- (3.5,1);
\draw[color=blue] (2.5,-1.5) -- (3.5,-1) -- (2.5,-0.5);
\draw[color=ForestGreen] (2.5,0.5) -- (3.5,1) -- (2.5,1.5);
\end{scope}
\end{tikzpicture}
\caption{Two elements of $SV_G$, for $S=\{r,b,g\}$ (represented by red, blue, and green), $g_1,\dots,g_4\in G$, and $k_1,\dots,k_4\in \ker(G\curvearrowright S)$. The element on the right is in $SK_G$ since the two multicolored trees are the same and the labels from $G$ lie in $\ker(G\curvearrowright S)$.}
\label{fig:elements}
\end{figure}
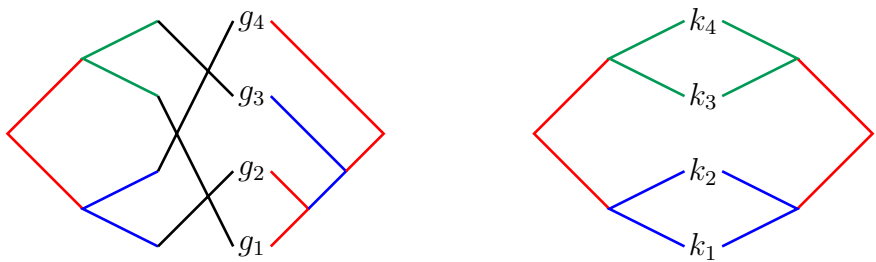

\begin{remark}
Setting $K=\ker(G\curvearrowright S)$, there is a natural homomorphism $SV_G \to SV_{G/K}$, defined by
\[[T_-,\sigma,(g_1,\dots,g_n),T_+] \mapsto [T_-,\sigma,(g_1K,\dots,g_nK),T_+].\]
This is well-defined because $K$ acts trivially on $S$. The kernel is precisely $SK_G$, which shows in particular that $SK_G$ is a normal subgroup of $SV_G$.
\end{remark}

\begin{remark}
The reader familiar with so-called cloning systems, as in \cite{witzel18,zaremsky18user}, may notice that (abstract) twisted Brin--Thompson groups could serve as examples of a ``multicolored'' generalization of cloning systems, likely also including the braided Brin--Thompson groups from \cite{spahn_phd}. This is beyond the scope of the present work, but could be an interesting proposed framework to investigate further.
\end{remark}

Abstract twisted Brin--Thompson groups simultaneously generalize two existing families of groups. First, if $G\curvearrowright S$ is faithful, then $SV_G$ is a \emph{twisted Brin--Thompson group}, which here we may call a \emph{faithful twisted Brin--Thompson group}. This is a genuine group of homeomorphisms of $C^S$, with the entry $(g_1,\dots,g_n)$ encoding the $g_i\in G$ acting on dyadic bricks by permuting coordinates. Second, if $S=\{s\}$ is a single point, so $G\curvearrowright \{s\}$ is trivial, then $\{s\}V_G$ is the \emph{labeled Thompson group} $V(G)$; this was essentially introduced by Thompson in \cite{thompson80}, and see also \cite[Subsection~1.5]{wu25} for a more modern treatment.

\medskip

One more important tool we will need is the notion of a germinal twist, generalizing the definition in the faithful case from \cite{belk22}.

\begin{definition}[Germinal twist]
Let $h=[F_-,\sigma,(g_1,\dots,g_n),F_+]$ be an element of $S\mathcal{V}_G$ and let $\kappa\in C^S(m)$, where $m$ is the corank of $F_+$. The \emph{germinal twist} $\gtwist_\kappa(h)$ of $h$ at $\kappa$ is the element $g_i$ of $G$ such that $F_+(\kappa)$ lies in the $i$th leaf of $F_+$.
\end{definition}

Note that $\gtwist_\kappa(h)$ respects the defining relations of the groupoid $S\mathcal{V}_G$, and so is well defined. Also note that
\begin{equation}\label{eq:twist_two}
\gtwist_\kappa(h'h)=\gtwist_{h.\kappa}(h')\gtwist_\kappa(h)
\end{equation}
for all $h,h'\in SV_G$ and all $\kappa\in C^S$.

Just to give an example, let $h$ be the element on the left in Figure~\ref{fig:elements} and $\kappa$ any point such that $01$ is a prefix of $\kappa(r)$ and $0$ is a prefix $\kappa(b)$. Then (using the convention that bottom to top in the picture corresponds to left to right in the tuple of elements of $G$) we get $\gtwist_\kappa(h)=g_2$.

\section{Relative simplicity}\label{sec:simple}

Recall from the introduction that a group is relatively simple if it has a proper normal subgroup, called the largest, containing every proper normal subgroup. In this section we look at some examples and non-examples, and then prove Theorem~\ref{thrm:main:rel_simple}.

\subsection{Examples and non-examples}

The first example of relatively simple groups is simple groups, with largest normal subgroup the trivial one. Beyond simple groups, there are many existing examples of groups from the Thompson world that are relatively simple but not simple. For example, in the braided Thompson group $bV$ \cite{brin07,dehornoy06} and the labeled Thompson groups $V(G)$ \cite{thompson80}, every proper normal subgroup lies in the kernel of the canonical map to $V$ \cite{zaremsky18,wu25}.

Central extensions of simple groups are often relatively simple. For example, for even $n \geq 2$ and most fields $K$, the simple \cite[Theorems 8.14 and 8.23]{rotman} group $PSL_n(K)$ has as a central extension $SL_n(K)$, which is relatively simple with largest normal subgroup $\{\pm I\}$. In a different context, in the lift $\overline{T}$ of Thompson's group $T$ to the line \cite{ghys87}, every proper normal subgroup lies in the kernel of the canonical map to $T$ \cite[Lemma~3.1]{fournierfacio24}; other related examples are discussed further in \cite{fournierfacio24}.

Examples also arise from the world of lamplighter groups, or permutational wreath products.

\begin{lemma}\label{lem:lamp}
Let $B$ be a perfect group and let $A$ be a simple group acting non-trivially on a set $S$. Then the permutational wreath product $B \wr_S A = \bigoplus_S B \rtimes A$ is relatively simple, with largest normal subgroup $\bigoplus_S B$.
\end{lemma}

\begin{proof}
Let $(\vec{b},a)$ be an element of $B\wr_S A$ with $a\ne 1$, and we need to prove that its normal closure is the whole group. Since the action of $A$ on $S$ is non-trivial and $A$ is simple, the action is faithful. Since $a\ne 1$ we can therefore choose $s\in S$ such that $a.s\ne s$. For $b\in B$ and $t\in S$, write $b_t$ for the element of $\bigoplus_S B$ that is $b$ in the $t$-th copy of $B$. Note that $b_s$ and $c_{a.s}$ commute for all $b,c\in B$. Now we compute the following double commutator for arbitrary $b,c\in B$, which is an element of the normal closure of $(\vec{b},a)$:
\begin{align*}
[b_s,[c_s^{-1},(\vec{b},a)]] &= [b_s, c_s (c_s^{-1})^{(\vec{b},a)}] \\
&= [b_s, c_s ((c^{-1})^{\vec{b}_s})_{a.s}] \\
&= [b_s,c_s]\text{.}
\end{align*}
This shows that the normal closure of $(\vec{b},a)$ contains $\bigoplus_S [B,B]$, and since $B$ is perfect this is $\bigoplus_S B$. Now since $A$ is simple the result follows.
\end{proof}

Before turning to non-examples, let us give a characterization of relative simplicity. (Here we continue to follow the convention that the trivial group is not simple.)

\begin{lemma}\label{lem:unique}
A finitely generated group $G$ is relatively simple if and only if $G$ has a unique normal subgroup with simple quotient.
\end{lemma}

\begin{proof}
Suppose that $(G, N)$ is a finitely generated relatively simple pair. Because there is no intermediate normal subgroup $N < M < G$, the group $G/N$ is simple. Moreover, if $G/M$ is another simple quotient of $G$, then $M$ is contained in $N$, hence $G/M$ surjects onto $G/N$, so by simplicity $M = N$.

Next, suppose that $N < G$ is the unique normal subgroup such that $G/N$ is simple, and let $M < G$ be a proper normal subgroup of $G$. Recall that every finitely generated group has a simple quotient; this is obvious if $G$ has a finite quotient, and otherwise it follows from a Zorn's lemma argument that goes back at least to Higman \cite{higman:group}. In particular $G/M$ has a simple quotient, which is then also a simple quotient of $G$, hence it must be $G/N$. This shows that $M < N$ and concludes the proof.
\end{proof}

\begin{remark}
We only used the assumption that $G$ is finitely generated in the second part, so it is still true that a relatively simple group has a unique simple quotient. The converse is not true however: the group $\Q$ has no simple quotient, which shows that e.g., $\Q \times \Z/2$ is a group with a unique normal subgroup with simple quotient (namely $\Q\times\{0\}$), but which is not relatively simple.
\end{remark}

We should also mention another prominent class of simple groups, which likely have relatively simple generalizations, namely commutator subgroups of Scott--R\"over--Nekrashevych groups, first developed by Scott in \cite{scott84}, and brought to prominence by R\"over \cite{roever99} and Nekrashevych \cite{nekrashevych04}. Here the input is a faithful self-similar action of a group $G$ on a rooted $d$-regular tree, and the output is a group denoted $V_d(G)$, whose commutator subgroup is always simple. If we relax the assumption that the self-similar action of $G$ be faithful, then we still get a group $V_d(G)$, and presumably now its commutator subgroup is relatively simple instead of simple. It is not clear to us however what sort of utility these groups have for investigating any Boone--Higman-related questions, so we will not work out any details here.

\medskip

Let us also mention some broad classes of non-examples of relative simplicity.

\begin{example}\label{ex:surj:Z}
If $G$ is a group surjecting onto $\Z$, then it is never relatively simple, because it has distinct simple quotients: $\Z/2$ and $\Z/3$, for example. Similarly a group that surjects onto a finite abelian group that is not cyclic of prime power order, is not relatively simple (note that cyclic groups of prime power order \emph{are} relatively simple).

Hence the abelianization of a finitely generated relatively simple group must be cyclic of prime power order. Note that it does not need to be trivial, as witnessed by some Higman--Thompson groups $V_n$ \cite{higman74}.
\end{example}

\begin{example}
If $G$ is acylindrically hyperbolic, then $G$ admits an infinite simple quotient, and in fact we can prescribe the quotient to be injective on any given finite set, after modding out the finite radical \cite[Theorem~3.5]{characteristic}. This implies that $G$ has infinitely many simple quotients with distinct kernel, hence $G$ is not relatively simple. In particular this applies to non-elementary hyperbolic and relatively hyperbolic groups, mapping class groups of non-exceptional surfaces, and outer automorphism groups of finitely generated non-abelian free groups \cite{osin:AH}.
\end{example}

\subsection{Relative simplicity of abstract twisted Brin--Thompson groups}

The main result of this section is the following.

\begin{theorem}[Theorem~\ref{thrm:main:rel_simple}]
\label{thrm:rel_simple}
Let $G$ be a group acting on a non-empty set $S$. The abstract twisted Brin--Thompson group $SV_G$ is relatively simple, with largest normal subgroup the canonical kernel $SK_G$.
\end{theorem}

Before proving this we need to collect some preliminaries. First note that $SK_G$ is isomorphic to a direct limit of groups of the form $K\times\cdots\times K$. Indeed, for a fixed multicolored tree $T$ of rank $n$, the elements of the form $[T,\id,(k_1,\dots,k_n),T]$ form a subgroup of $SK_G$ isomorphic to $K^n$. This description of the largest normal subgroup of $SV_G$ has some similarities to the description of the largest normal subgroup of the braided Thompson group $bV$, which is a direct limit of pure braid groups subject to bifurcations of strands; see \cite{zaremsky18} for details.

\begin{definition}[Deferment]
Let $B$ be a dyadic brick and $g\in G$. Write $D_B(g)$ for the \emph{deferment} of $g$ to $B$, defined to be
\[
D_B(g) \coloneqq [T,\id,(1,\dots,1,g,1,\dots,1),T]\text{,}
\]
where $T$ is some (any) multicolored tree with a leaf corresponding to $B$, say the $i$th leaf, and $g$ is in the $i$th position in the tuple.
\end{definition}

Intuitively, $D_B(g)$ ``does'' $g$ to the dyadic brick $B$ and does nothing outside $B$; see Figure~\ref{fig:defer} for an example. Note that $SK_G$ is generated by all the deferments of elements of $K$ to proper dyadic bricks.

\begin{figure}[htb]
\centering
\begin{tikzpicture}[line width=1pt]
\draw[color=red] (2,-2) -- (0,0) -- (1,1);
\draw[color=blue] (2,0) -- (1,1) -- (2,2);
\node at (2.25,-2) {$1$};
\node at (2.25,0) {$g$};
\node at (2.25,2) {$1$};
\draw[color=red] (2.5,-2) -- (4.5,0) -- (3.5,1);
\draw[color=blue] (2.5,0) -- (3.5,1) -- (2.5,2);
\end{tikzpicture}
\caption{With $S=\{r,b\}$, the deferment $D_{B(\psi)}(g)$ for $\psi(r)=1$ and $\psi(b)=0$.}
\label{fig:defer}
\end{figure}
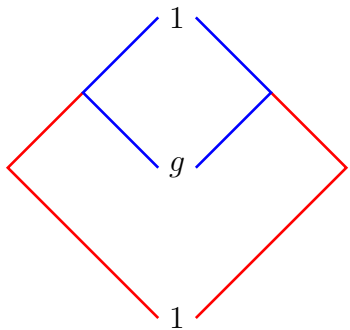

\begin{proof}[Proof of Theorem~\ref{thrm:rel_simple}]
The proof is inspired by the proof of Theorem~2.1 of \cite{wu25} for the labeled Thompson groups, i.e., the $|S|=1$ case. Let $N$ be a normal subgroup of $SV_G$ not contained in $SK_G$, and we must show that $N=SV_G$. Since $SK_G$ is the kernel of the action of $SV_G$ on $C^S$, we can choose an element $h$ of $N$ that acts non-trivially on $C^S$. We claim that we can choose a dyadic brick $B(\psi)$, such that $h(B(\psi))=B(\varphi)$ for some $\varphi$, $B(\psi)\cap B(\varphi)=\emptyset$, and $B(\psi)\cup B(\varphi)\ne C^S$. Indeed, the dyadic bricks form a basis of the Hausdorff space $C^S$, so we know we can choose $\psi$ so as to ensure $B(\psi)\cap h(B(\psi))=\emptyset$, and up to replacing $B(\psi)$ with a smaller dyadic brick, we can also achieve that the image $h(B(\psi))$ is itself a dyadic brick $B(\varphi)$ and $B(\psi)\cup B(\varphi)\ne C^S$.

Let $f$ be the deferment $D_{B(\psi)}(k)$ for $k\in K$. Then the commutator $hfh^{-1}f^{-1}$ equals $D_{B(\varphi)}(gkg^{-1}) D_{B(\psi)}(k)^{-1}$ for some $g\in G$. Since $h\in N$, this commutator is also in $N$. Conjugating by $D_{B(\varphi)}(g)$, which commutes with $D_{B(\psi)}(k)$, we can assume without loss of generality that $g=1$, so $D_{B(\varphi)}(k) D_{B(\psi)}(k)^{-1}$ lies in $N$.

Now let $A,A'$ be any pair of disjoint dyadic bricks whose union is not all of $C^S$. We can choose an element of $SV$ taking $B(\psi)$ to $A$ via the canonical homeomorphism and taking $B(\varphi)$ to $A'$ via the canonical homeomorphism. Since $D_{B(\varphi)}(k) D_{B(\psi)}(k)^{-1}\in N$, we get that $D_{A'}(k) D_A(k)^{-1} \in N$ for all such $A,A'$ and all $k\in K$. Taking products of elements of this form, we get that $D_{A'}(k) D_A(k)^{-1} \in N$ for all proper dyadic bricks $A,A'$ and all $k\in K$. Finally note that for any $B(\psi)$, $k\in K$, and $s\in S$, $D_{B(\psi)}(k) = D_{B(\psi_0^s)}(k) D_{B(\psi_1^s)}(k)$, so upon modding out $N$, for any proper dyadic brick $A$ and any $k\in K$, we get that $D_A(k)$ is identified with $D_A(k) D_A(k)$. We conclude that every such $D_A(k)$ lies in $N$, and since they generate $SK_G$ in fact all of $SK_G$ lies in $N$. Now simplicity of the faithful twisted Brin--Thompson group $SV_G/SK_G = SV_{G/K}$ \cite[Theorem~3.4]{belk22} plus the fact that $N$ does not equal $SK_G$ tells us that $N=SV_G$.
\end{proof}

\section{Finiteness properties}\label{sec:fin_props}

In this section we establish when an abstract twisted Brin--Thompson group $SV_G$ is finitely generated, finitely presented, and type~$\F_\infty$. These were all previously worked out for faithful twisted Brin--Thompson groups in \cite{belk22, zaremsky_fp_tbt}; the main result of this section is that the analogous results hold in the non-faithful case.

\begin{theorem}[Theorem~\ref{thrm:main:fp}]
\label{thrm:fin_props}
Let $G$ be a group acting on a non-empty set $S$. Then:
\begin{enumerate}
    \item $SV_G$ is finitely generated if and only if $G\curvearrowright S$ is of type~$[\A_1]$.
    \item $SV_G$ is finitely presented if and only if $G\curvearrowright S$ is of type~$[\A_2]$.
    \item $SV_G$ is of type~$\F_\infty$ if and only if $G\curvearrowright S$ is of type~$[\A_\infty]$.
\end{enumerate}
\end{theorem}

The most interesting and important of these is item (ii), as Section~\ref{sec:embed} will make clear. The rest of this section is devoted to proving Theorem~\ref{thrm:fin_props}. First let us record the following easy consequence of item (ii).

\begin{corollary}\label{cor:push_fin_props}
Let $G$ be a group acting on a non-empty set $S$. Let $K=\ker(G\curvearrowright S)$. If $SV_G$ is finitely presented, then the simple group $SV_{G/K}$ is finitely presented if and only if $G/K$ is.
\end{corollary}

\begin{proof}
Suppose $SV_G$ is finitely presented, so the action $G\curvearrowright S$ is of type~$[\A_2]$. Hence $G$ is finitely presented, each $\Stab_G(s)$ is finitely generated, and $G\curvearrowright S\times S$ has finitely many orbits. Passing to the quotient $G/K$ acting on $S$, we have that the point stabilizers are still finitely generated and there are still finitely many orbits in $S\times S$, so $G/K \curvearrowright S$ is of type~$(\A_2)$ if and only if $G/K$ is finitely presented. But $SV_{G/K}$ being finitely presented is also equivalent to $G/K \curvearrowright S$ being of type~$(\A_2)$ (by Theorem~\ref{thrm:fin_props}(ii), or by \cite{zaremsky_fp_tbt}), and we are done.
\end{proof}

\subsection{Generators}\label{ssec:gens}

In this subsection we pin down an infinite generating set for $SV_G$, and prove that $SV_G$ is finitely generated if and only if $G\curvearrowright S$ is of type~$[\A_1]$. This will proceed similarly to the faithful case from \cite[Section~3]{belk22}. Throughout the subsection we fix a group $G$ acting on a non-empty set $S$.

Let $\iota_\varnothing\colon G\to SV_G$ send $g$ to $[\cdot,\id,g,\cdot]$, where $\cdot$ is the trivial tree. For $s\in S$, let $\iota_1^s\colon G\to SV_G$ send $g$ to $[x_s,\id,(1,g),x_s]$.

\begin{lemma}\label{lem:first_gens}
For any fixed $s\in S$, $SV_G$ is generated by $SV\cup \iota_1^s(G)$.
\end{lemma}

\begin{proof}
For $[T_-,\sigma,(g_1,\dots,g_n),T_+]\in SV_G$, up to expansions and multiplying on either side by elements of $SV$ we can assume without loss of generality that $T_-=T_+$, $n\ge 2$, and $\sigma=\id$. Now our element can be written as a product of conjugates (by elements of $SV$) of elements of the form $[T,\id,(1,\dots,1,g),T]$. Again up to multiplying by elements of $SV$, we can assume $T=(x_s\oplus 1_{n-1})\circ\cdots\circ (x_s\oplus 1)\circ x_s$. But now reversing the expansion operation we get $[x_s,\id,(1,g),x_s]$, which is in $\iota_1^s(G)$, and we are done.
\end{proof}

For $S'\subseteq S$ write $S'V^{(S)}$ for the subgroup of $SV$ consisting of all $[T_-,\sigma,T_+]$ such that $T_-$ and $T_+$ are $S'$-multicolored.

\begin{corollary}\label{cor:good_gens}
Let $\Delta(S)$ be a connected graph with vertex set $S$ such that the action of $G$ on $S$ extends to an action on the graph. Let $\{e_\alpha\}_{\alpha\in\mathcal{I}}$ be a set of representatives of the edge orbits of $G\curvearrowright \Delta(S)$. If $|S|\ge 2$ then for any fixed $s\in S$, $SV_G$ is generated by the subgroups $e_\alpha V^{(S)}$ (each of which is isomorphic to $2V$) together with $\iota_\varnothing(G)$ and $\iota_1^s(G)$.
\end{corollary}

\begin{proof}
By Lemma~\ref{lem:first_gens} it suffices to prove that $SV$ lies in the subgroup generated by the $e_\alpha V^{(S)}$ and $\iota_\varnothing(G)$. Since every edge of $\Delta(S)$ is a $G$-translate of some $e_\alpha$, and conjugation by $\iota_\varnothing(g)$ takes $e_\alpha V^{(S)}$ to $(g.e_\alpha)V^{(S)}$, we know this subgroup contains $eV^{(S)}$ for every edge $e$ of $\Delta(S)$. Now \cite[Proposition~3.2]{belk22} says this subgroup contains $SV$, and we are done.
\end{proof}

If $|S|=1$, say $S=\{s\}$, then $\Delta(S)=S$ has no edges, so we cannot phrase it quite the same, but in this case $SV\cong V$ is finitely generated anyway, so the generating set from Lemma~\ref{lem:first_gens} is already useful.

Now we can prove the finite generation part of Theorem~\ref{thrm:fin_props}.

\begin{proof}[Proof of Theorem~\ref{thrm:fin_props}(i)]
First suppose $G\curvearrowright S$ is of type~$[\A_1]$, i.e., $G$ is finitely generated and $G\curvearrowright S$ has finitely many orbits. If $|S|=1$ then by Lemma~\ref{lem:first_gens} $SV_G$ is generated by a copy of $V$ and a copy of $G$, hence is finitely generated, so assume $|S|\ge 2$. Fix a set $\{s_1,\dots,s_m\}$ of representatives of the $G$-orbits in $S$, and fix a finite generating set $A=\{a_1,\dots,a_k\}$ for $G$. Construct a graph $\Delta(S)$ to have vertex set $S$, an edge from $s_1$ to $s_i$ for each $2\le i\le m$, an edge from $s_1$ to $a_i.s_1$ for each $1\le i\le k$ such that $a_i.s_1 \ne s_1$, and an edge for each $G$-translate of either of those types of edges. This graph is connected and $G$-invariant by construction, with finitely many orbits of edges. By Corollary~\ref{cor:good_gens}, $SV_G$ is generated by finitely many copies of $2V$ (which is finitely generated \cite[Proposition~6.2]{brin04}) plus two copies of $G$, and so we conclude it is finitely generated.

Now suppose $SV_G$ is finitely generated, and we must show $G\curvearrowright S$ is of type~$[\A_1]$. Setting $K=\ker(G\curvearrowright S)$, the quotient $SV_{G/K}$ is finitely generated, so by the faithful case (\cite[Theorem~A]{belk22}) $(G/K)\curvearrowright S$ is of type~$(\A_1)$. In particular this action has finitely many orbits, and thus so does $G\curvearrowright S$. It remains to prove that $G$ is finitely generated. Let $G_1\le G_2\le \cdots$ be an arbitrary sequence of subgroups of $G$ whose union is $G$, so $SV_{G_1}\le SV_{G_2}\le \cdots$ is a sequence of subgroups of $SV_G$ whose union is $SV_G$. Since $SV_G$ is finitely generated, there exists $m$ such that $SV_{G_m}=SV_{G_{m+1}}=\cdots$. Let $g\in G_{m+1}$, and let $h=[\cdot,\id,g,\cdot]$. Then $h\in SV_{G_{m+1}}$, hence $h\in SV_{G_m}$, and so every germinal twist $\gtwist_\kappa(h)$ for $\kappa\in C^S$ lies in $G_m$. But $\gtwist_\kappa(h)=g$ for all $\kappa$, so $g\in G_m$. This shows $G_m=G_{m+1}$, and iterating this argument we conclude that $G_m=G_{m+1}=\cdots$. Since this sequence was arbitrary, $G$ is finitely generated.
\end{proof}

\begin{remark}\label{rmk:complete_graph}
If $G\curvearrowright S$ has finitely many orbits of pairs, then we can simply take $\Delta(S)$ to be the complete graph on $S$ (meaning every pair of elements of $S$ spans an edge), and extract a finite generating set for $SV_G$ out of finitely many copies of $2V$ and $G$.
\end{remark}

\subsection{Actions on complexes}\label{ssec:positive}

In this section we prove the ``if'' direction of items (ii) and (iii) of Theorem~\ref{thrm:fin_props}, that if $G\curvearrowright S$ is nice then the finiteness properties of $SV_G$ are nice. We will use the action of $SV_G$ on a certain complex associated to $SV_{G/K}$ coming from the quotient map $SV_G\to SV_{G/K}$. This complex, used in \cite{zaremsky_fp_tbt}, is a certain subcomplex of the so-called Stein complex from \cite{belk22}; the name follows the analogous complex for Thompson's group $F$ constructed by Stein in \cite{stein92}. We will define the Stein complex $SX_G$ in general (i.e., allowing $G\curvearrowright S$ to be non-faithful), but for now we will not prove all the expected connectivity and contractibility results that have already been done in the faithful case. Indeed this would be very involved, and we do not need to since it turns out the action of $SV_G$ on $SX_{G/K}$ will reveal everything we want to know here.

We reiterate that our definition of multicolored forest does not include permutations, whereas the definition in \cite{belk22,zaremsky_fp_tbt} does. In practice this is irrelevant since we will often be looking at cosets in $S\mathcal{V}$ of the form $\mathcal{S}(n)F$ or $\mathcal{W}(n)F$, for $F$ a multicolored forest (with rank $n$), where recall that $\mathcal{S}(n)$ is the copy of $\Sigma_n$ permuting the terms of $C^S(n)$ and $\mathcal{W}(n)$ is the copy of $\Sigma_n \ltimes G^n$ generated by $\mathcal{S}(n)$ and $\mathcal{G}(n)$.

\begin{definition}[Elementary, spectrum]
We define what it means for a multicolored forest to be \emph{elementary}, and what its \emph{spectrum} is, as follows. The identity $1_m$ is elementary for all $m$, with spectrum $\emptyset$. Now given a multicolored forest $F$ with spectrum $S'\subseteq S$, say $F$ has corank $m\ge 2$ and rank $n$, and given an element $s\in S\setminus S'$, we declare that the multicolored forest $F\circ (1_i\oplus x_s\oplus 1_{m-i-2})$ ($0\le i\le m-2$) is elementary with spectrum $S'\cup\{s\}$. Call an elementary multicolored forest \emph{$k$-elementary} if its spectrum has size at most $k$.
\end{definition}

Intuitively, a multicolored forest is elementary if, viewed as a homeomorphism, each input is ``sent through'' at most one simple split $x_s$ of each color $s\in S$. The spectrum is the set of colors that are actually used. For example, returning to Figures~\ref{fig:MF} and~\ref{fig:cross}, the multicolored tree in Figure~\ref{fig:MF} is not elementary since some inputs are sent through two copies of $x_r$, and the multicolored tree in Figure~\ref{fig:cross} is elementary with spectrum $\{r,b\}$.

\begin{definition}[Stein complex, short/long edge]\label{def:stein}
Let $G$ be a group acting faithfully on a set $S$. The \emph{Stein complex} $SX_G$ is the flag simplicial complex constructed as follows. The vertices are the cosets $\mathcal{W}(m)h$ for $h\in S\mathcal{V}_G$ an element with corank~$1$ and rank~$m$. There is an edge from $\mathcal{W}(m)h$ to $\mathcal{W}(n)(F\circ h)$ for each elementary multicolored forest with corank $m$ and rank $n$. The complex $SX_G$ is now the flag complex with this $1$-skeleton. Call an edge \emph{short} if the ranks of its endpoints differ by $1$, and \emph{long} otherwise. Let $SX_G(k)$ be the flag subcomplex of $SX_G$ with the same vertex set but only including the edges from $\mathcal{W}(m)h$ to $\mathcal{W}(n)(F\circ h)$ where $F$ is $k$-elementary. Finally, let $SX_G^n$ and $SX_G^n(k)$ be the full subcomplexes of $SX_G$ and $SX_G(k)$, respectively, spanned by all vertices with rank at most~$n$.
\end{definition}

There is a (right) action of $SV_G$ on $SX_G$ via right multiplication in $S\mathcal{V}_G$. This action stabilizes each of the subcomplexes defined by restricting the ranks of vertices and/or sizes of spectra of multicolored forests being used to define edges.

For $G\curvearrowright S$ faithful, in \cite{belk22,zaremsky_fp_tbt} some results were obtained on the connectivity of these complexes. The following collects together these results, phrased in a convenient way.

\begin{proposition}[Connectivity]\label{prop:stein_conn}
Assume $G\curvearrowright S$ is faithful. Then the following hold.
\begin{enumerate}
    \item The complex $SX_G$ is contractible.
    \item For all $n\in\N$ there exists $M\in\N$ such that for all $m\ge M$ the complex $SX_G^m$ is $(n-1)$-connected.
    \item For all $k\ge n$ the complex $SX_G(k)$ is $(n-1)$-connected.
    \item For all $k\ge n$ and all $m$, if $SX_G^m$ is $(n-1)$-connected then so is $SX_G^m(k)$.
\end{enumerate}
\end{proposition}

\begin{proof}
Item (i) is \cite[Proposition~5.6]{belk22}. Item (ii) follows by combining \cite[Proposition~7.9]{belk22} (which concerns descending links with respect to the Morse function induced by the rank) with standard discrete Morse theory, e.g., \cite[Corollary~2.6]{bestvina97}. Items (iii) and (iv) are \cite[Proposition~4.4]{zaremsky_fp_tbt}.
\end{proof}

Again, it is reasonable to expect that these connectivity bounds also hold when $G\curvearrowright S$ is not faithful, but proving this would be a large digression with no benefit to our present purposes, so we leave this alone for now. In particular, for $3\le n<\infty$, even in the faithful case, $G\curvearrowright S$ being of type~$[\A_n]$ is not sufficient for the action of $SV_G$ on $SX_G^m(k)$ to reveal that $SV_G$ is of type~$\F_n$ (as conjectured), since not all simplex stabilizers have the ``right'' finiteness properties. Thus we believe that some different variation of the Stein complex could be needed anyway, in both the faithful and non-faithful cases, to handle type $\F_n$ for $3\le n<\infty$.

\medskip

Now we return to $G\curvearrowright S$ possibly being non-faithful, with kernel $K$. The action of $SV_{G/K}$ on $SX_{G/K}$ and the map $SV_G\to SV_{G/K}$ give us an action of $SV_G$ on $SX_{G/K}$. We will now exploit this to deduce information about finiteness properties.

In the wreath product $\Sigma_n \ltimes G^n$, for $1\le k\le n$ let $\Sigma_{n-1}^{(k)}\coloneqq \Stab_{\Sigma_n}(k)$, and for $s\in S$ consider the subgroup
\[
E_n^k(s)\coloneqq \Sigma_{n-1}^{(k)}\ltimes (G\times\cdots\times G\times \Stab_G(s)\times G\times\cdots\times G)
\]
of $\Sigma_n \ltimes G^n$, where the $\Stab_G(s)$ factor appears in the $k$th spot.

\begin{lemma}[Stabilizers]\label{lem:stabs}
With notation as above, consider the action of $SV_G$ on $SX_{G/K}$. The following hold.
\begin{enumerate}
    \item For any vertex $x$ of $SX_{G/K}$, say with rank $n$, the stabilizer in $SV_G$ of $x$ is conjugate in $S\mathcal{V}_G$ to $\mathcal{W}(n)\cong \Sigma_n \ltimes G$.
    \item For any short edge $e$ in $SX_{G/K}$, say with lower-rank endpoint $x$ of rank $n$, the stabilizer of $e$ in $SV_G$ fixes $x$ and maps under the above isomorphism to $E_n^k(s)$ for some $k$ and $s$.
    \item For any simplex $c$ in $SX_{G/K}$, say with minimal-rank vertex $x$ of rank $n$, the stabilizer of $c$ in $SV_G$ fixes $x$ and maps under the above isomorphism to a subgroup of $\Sigma_n \ltimes G^n$ commensurable to
    \[
    \Stab_G(S_1)\times\cdots\times\Stab_G(S_n)
    \]
    for some finite $S_1,\dots,S_n\subseteq S$.
\end{enumerate}
\end{lemma}

\begin{proof}
First let us prove all the results in the faithful case (which was basically done in \cite{belk22,zaremsky_fp_tbt}, though not in exactly this form), and then deduce them in the non-faithful case.

Assume $G\curvearrowright S$ is faithful. Let $x=\mathcal{W}(n)h$ be a vertex of $SX_G$ with rank $n$. An element of $SV_G$ stabilizes $\mathcal{W}(n)h$ if and only if its conjugate via $h$ lies in $\mathcal{W}(n)$, proving (i). Now consider a short edge from $x$ to a rank-$(n+1)$ vertex, say $\mathcal{W}(n+1)(F\circ h)$. Since $F$ has rank $n+1$ and corank $n$, it is of the form $1_{k-1}\oplus x_s\oplus 1_{n-k}$ for some $1\le k\le n$ and some $s\in S$. The stabilizer of this edge fixes each vertex (since the action preserves rank), and conjugating by $h$ it is isomorphic to the subgroup of $\mathcal{W}(n)\cong \Sigma_n \ltimes G^n$ consisting of all elements $(\sigma,(g_1,\dots,g_n))$ such that $\sigma$ fixes $k$ and $g_k$ fixes $s$. This is precisely $E_n^k(s)$, proving (ii). Finally, the stabilizer of any simplex is commensurable to the stabilizer of its edge whose endpoints have the largest and smallest rank. Thus we may assume we are dealing with a single edge, say from $\mathcal{W}(n)h$ to $\mathcal{W}(m)(F\circ h)$. Say $F=T_1\oplus\cdots\oplus T_n$ for $T_i$ multicolored trees. After conjugating by $h$, the stabilizer of this edge becomes a subgroup of $\Sigma_n \ltimes G^n$ commensurable with the subgroup of $G^n$ consisting of all $(g_1,\dots,g_n)$ such that $g_i$ fixes the spectrum of $T_i$, for each $i$. This proves (iii).

Now suppose $K=\ker(G\curvearrowright S)$ can be non-trivial. Viewing $\mathcal{W}(n)$ as $\Sigma_n \ltimes G^n$, the quotient $S\mathcal{V}_G \to S\mathcal{V}_{G/K}$ takes $\Sigma_n \ltimes G^n$ to $\Sigma_n \ltimes (G/K)^n$, and the action of $SV_G$ on $SX_{G/K}$ is induced by $S\mathcal{V}_G \to S\mathcal{V}_{G/K}$, so all the results are immediate from the faithful case.
\end{proof}

We can now prove the $\F_\infty$ statement.

\begin{proposition}\label{prop:yes_Finfty}
If $G\curvearrowright S$ is of type~$[\A_\infty]$ then $SV_G$ is of type~$\F_\infty$.
\end{proposition}

\begin{proof}
Consider the action of $SV_G$ on the complex $SX_{G/K}$. This is contractible by Proposition~\ref{prop:stein_conn}(i), and the filtration $SX_{G/K}^m$ is $SV_G$-invariant. For all $n$ there exists $m$ such that $SX^m_{G/K}$ is $(n-1)$-connected, by Proposition~\ref{prop:stein_conn}(ii). Now by \cite[Theorem~7.3.1]{geoghegan08}, it suffices to prove that each $SX_{G/K}^m$ is $SV_G$-cocompact, and the stabilizer in $SV_G$ of each simplex in $SX_{G/K}$ is of type~$\F_\infty$.

Since $G\curvearrowright S$ is of type~$[\A_\infty]$, it is oligomorphic, and so $(G/K)\curvearrowright S$ is oligomorphic as well. Thus, by \cite[Proposition~6.7]{belk22} the action of $SV_{G/K}$ on $SX_{G/K}^m$ is cocompact, and so the action of $SV_G$ is too.

The stabilizer in $SV_G$ of any simplex in $SX_{G/K}$ is commensurable by Lemma~\ref{lem:stabs}(iii) to some $\Stab_G(S_1)\times\cdots\times\Stab_G(S_n)$ for finite $S_1,\dots,S_n\subseteq S$. Since $G\curvearrowright S$ is of type~$[\A_\infty]$, these are all of type~$\F_\infty$.
\end{proof}

Now we focus on finite presentability. We will use the following criterion.

\begin{cit}\label{cit:fp_criterion}\cite[Section~3]{zaremsky_fp_tbt}
Let $X$ be a simply connected simplicial complex and $\Gamma$ a group with an orientation-preserving, cocompact action on $X$. Suppose that for each edge $e=\{v,w\}$ there exists an edge path $e_1,\dots,e_n$ from $v$ to $w$ such that each stabilizer $\Gamma_{e_i}$ is finitely generated, and the subgroup $\bigcap_{i=1}^n \Gamma_{e_i}$ of $\Gamma_e$ has finite index. Then $\Gamma$ is isomorphic to a free product of finitely many vertex stabilizers, modulo finitely many additional relations.
\end{cit}

\begin{proposition}\label{prop:extract_pres}
Suppose $G$ is finitely generated, and that $G\curvearrowright S$ has finitely many orbits of pairs and finitely generated point stabilizers. Then $SV_G$ is isomorphic to a group of the form $P/N$, where $P$ is a free product of finitely many groups each commensurable to a direct product of finitely many copies of $G$, and $N$ is finitely normally generated.
\end{proposition}

\begin{proof}
Let $K=\ker(G\curvearrowright S)$ and consider the action of $SV_G$ on $SX_{G/K}$. Since adjacent vertices have different ranks and the action preserves rank, the action is orientation-preserving. By Proposition~\ref{prop:stein_conn}(ii) we can choose $m$ such that $SX_{G/K}^m$ is simply connected, and by Proposition~\ref{prop:stein_conn}(iv) also $SX_{G/K}^m(2)$ is simply connected. Since $G\curvearrowright S$ has finitely many orbits of pairs, so does $(G/K)\curvearrowright S$. Thus by \cite[Lemma~4.5]{zaremsky_fp_tbt}, the action of $SV_{G/K}$ on $SX_{G/K}^m(2)$ is cocompact, and hence so is the action of $SV_G$. For any edge $e$ in $SX_{G/K}^m(2)$, by \cite[Lemma~4.9]{zaremsky_fp_tbt} the endpoints of $e$ can be connected by a path of short edges such that the fixer in $SV_{G/K}$ of this path has finite index in the stabilizer in $SV_{G/K}$ of $e$. Lifting to $SV_G$ via the quotient $SV_G\to SV_{G/K}$, the fixer in $SV_G$ of the path has finite index in the stabilizer in $SV_G$ of $e$. By Lemma~\ref{lem:stabs}(ii), the stabilizer in $SV_G$ of any short edge is commensurable to $\Stab_G(s)\times G^k$ for some $s\in S$ and $k\in\N$, which is finitely generated by hypothesis. This verifies all the hypotheses of Citation~\ref{cit:fp_criterion}, and we conclude that $SV_G$ is isomorphic to the free product of finitely many vertex stabilizers, modulo finitely many additional relations. By Lemma~\ref{lem:stabs}(i) each vertex stabilizer is isomorphic to $\Sigma_n \ltimes G^n$ for some $n$, hence is commensurable to $G^n$, and so we are done.
\end{proof}

\begin{corollary}\label{cor:yes_fp}
If $G\curvearrowright S$ is of type~$[\A_2]$ then $SV_G$ is finitely presented.
\end{corollary}

\begin{proof}
Since $G\curvearrowright S$ is of type~$[\A_2]$, the hypotheses of Proposition~\ref{prop:extract_pres} are met and moreover $G$ is finitely presented. The result is now immediate from Proposition~\ref{prop:extract_pres}.
\end{proof}

We can also formulate a homological version. Say that $G\curvearrowright S$ is of \emph{type~$[\HA_2]$} if $G$ is of type~$\FP_2$, all point stabilizers are finitely generated, and there are finitely many orbits of pairs.

\begin{corollary}\label{cor:yes_FP2}
If $G\curvearrowright S$ is of type~$[\HA_2]$ then $SV_G$ is of type~$\FP_2$.
\end{corollary}

\begin{proof}
By Proposition~\ref{prop:extract_pres}, $SV_G$ is a quotient of a group of type~$\FP_2$ (namely a free product of finitely many groups each commensurable to a direct product of finitely many copies of $G$) modulo finitely many additional relations, so it suffices to prove that for any $\Gamma$ of type $\FP_2$, a quotient by finitely many additional relations is still of type $\FP_2$. By \cite[Lemma~2.1]{bieristrebel:FP2}, a group is of type $\FP_2$ if and only if it is a quotient of a finitely presented group modulo a perfect group, say $\Gamma=P/N$ for $P$ finitely presented and $N$ perfect. Now let $\overline{\Gamma}$ be a quotient of $\Gamma$ by finitely many additional relations. Lift these to $P$ to get a quotient $\overline{P}$ of $P$ by finitely many additional relations. Now $\overline{P}$ is still finitely presented, and the kernel of the quotient map $\overline{P}\to\overline{\Gamma}$ is the image of $N$ in $\overline{P}$, which is perfect since $N$ is perfect. We conclude that $\overline{\Gamma}$ is a quotient of a finitely presented group by a perfect normal subgroup, hence is of type~$\FP_2$ again by \cite[Lemma~2.1]{bieristrebel:FP2}.
\end{proof}

Finally, we can relax ``finitely presented'' to ``recursively presented''.

\begin{corollary}\label{cor:rec_pres}
Suppose $G$ is finitely generated, and that $G\curvearrowright S$ has finitely many orbits of pairs and finitely generated point stabilizers. Then $SV_G$ is recursively presented if and only if $G$ is.
\end{corollary}

\begin{proof}
Suppose $G$ is recursively presented. By Proposition~\ref{prop:extract_pres}, $SV_G$ is a quotient of a recursively presented group (namely a free product of finitely many groups each commensurable to a direct product of finitely many copies of $G$) modulo finitely many additional relations, hence is recursively presented. Conversely, if $SV_G$ is recursively presented then so is $G$, since $G$ is finitely generated and embeds as a subgroup of $SV_G$.
\end{proof}

Note that we cannot relax the hypothesis that $G\curvearrowright S$ has finitely many orbits of pairs to just finitely many orbits. Indeed, consider any finitely generated recursively presented group $G$ with unsolvable word problem. Using the action of $G$ on itself by translation, we get the finitely generated simple group $GV_G$, which contains $G$ and hence has unsolvable word problem. Since $GV_G$ is simple, it must therefore not be recursively presented, despite the fact that $G$ is recursively presented.

\begin{remark}
It would be interesting to try and use Corollary~\ref{cor:yes_FP2} to approach Boone--Higman-type questions regarding type $\FP_2$. For example, does there exist a type $\FP_2$ simple group with unsolvable word problem? (Equivalently a type $\FP_2$ simple group that is not recursively presented?) Do there exist uncountably many type~$\FP_2$ simple groups? Note that non-finitely presented type~$\FP_2$ simple groups exist \cite{LISW25}.
\end{remark}

\subsection{Quasi-retractions}\label{ssec:negative}

In this subsection we prove the ``only if'' direction of Theorem~\ref{thrm:fin_props}(ii), and in fact prove a version for higher finiteness properties. This works similarly to the case when $G\curvearrowright S$ is faithful, done in Section~2 of \cite{zaremsky_fp_tbt}. Recall that a \emph{quasi-retraction} from a metric space $X$ to a metric space $Y$ is a function $\rho\colon X\to Y$ such that there exists $\zeta\colon Y\to X$ with $\rho$ and $\zeta$ both coarse Lipschitz and $\rho\circ\zeta$ uniformly close to the identity on $Y$. If there exists a quasi-retraction from $X$ to $Y$, call $Y$ a \emph{quasi-retract} of $X$. Alonso proved in \cite{alonso94} that any quasi-retract of a group of type $\F_n$ is of type $\F_n$, and similarly for type $\FP_n$, with the groups viewed as metric spaces using word metrics coming from some finite generating sets.

Our strategy now is to prove that $\Z\wr_S G$ is a quasi-retract of $SV_G$. Since $\Z\wr_S G$ is of type $\F_n$ if and only if $G\curvearrowright S$ is of type~$[\A_n]$ \cite{cornulier06, bartholdi15}, this will prove that $G\curvearrowright S$ being of type~$[\A_n]$ is necessary for $SV_G$ to be type $\F_n$.

\begin{proposition}\label{prop:qr}
Let $G\curvearrowright S$ be a group action of type~$[\A_1]$, so $\Z \wr_S G$ and $SV_G$ are finitely generated. Then $\Z \wr_S G$ is a quasi-retract of $SV_G$.
\end{proposition}

\begin{proof}
Since $G\curvearrowright S$ is of type~$[\A_1]$, meaning $G$ is finitely generated and $G\curvearrowright S$ has finitely many orbits, we see that $\Z\wr_S G$ is finitely generated. Also, $SV_G$ is finitely generated by Theorem~\ref{thrm:fin_props}(i).

The construction of the quasi-retraction is essentially the same as in the faithful case in \cite{zaremsky_fp_tbt}, which in turn came from ideas of Jim Belk. Thus, we will move quickly through some of the steps that are identical to those in \cite{zaremsky_fp_tbt}. Let $h\in SV_G$ and $\kappa\in C^S$. Say $h=[T_-,\sigma,(g_1,\dots,g_n),T_+]$. Let $K=\ker(G\curvearrowright S)$, and consider the action of $SV_G$ on $C^S$ via the quotient $SV_G\to SV_{G/K}$. Let $B(\psi)$ be the dyadic brick containing $\kappa$ that is in the arboreal partition coming form $T_+$. Let $B(\varphi)$ be the dyadic brick in the arboreal partition coming from $T_-$ that $h$ sends $B(\psi)$ to via the canonical homeomorphism. For each $s\in S$, define $d_\kappa^s(h)$ to be the length of $\psi(s)$ in $\{0,1\}^*$, and $r_{h(\kappa)}^s(h)$ to be the length of $\varphi(s)$ in $\{0,1\}^*$. These depend on the choice of dyadic bricks, but for each $s\in S$ the integer
\[
r_{h(\kappa)}^s(h) - d_\kappa^{\gtwist_\kappa(h).s}(h)
\]
is well-defined on $h$. Define
\[
\rho_\kappa \colon SV_G \to \Z\wr_S G \qquad \text{via} \qquad h \mapsto ((r_{h(\kappa)}^s(h) - d_\kappa^{\gtwist_\kappa(h).s}(h))_{s\in S},\gtwist_\kappa(h))\text{.}
\]
By the same proof as in \cite[Lemma~2.5]{zaremsky_fp_tbt}, we have $\rho_\kappa(hh') = \rho_{h'(\kappa)}(h) \rho_\kappa(h')$ for all $h,h'\in SV_G$ and $\kappa\in C^S$.

Now fix a finite symmetric generating set $A$ of $SV_G$, and let $B$ be some finite generating set of $\Z\wr_S G$ that contains $\rho_\kappa(a)$ for all $a\in A$ and all $\kappa\in C^S$. (There really are only finitely many $\rho_\kappa(a)$ for a fixed $a$, since the value is constant as $\kappa$ varies over a given dyadic brick in a defining arboreal partition of the domain.) Let $\kappa_0\in C^S$ be the point defined by $\kappa_0(s)=\overline{0}$ for all $s\in S$, and we claim that $\rho_{\kappa_0}$ is a quasi-retraction. Here we use the word metrics coming from $A$ and $B$, and for convenience we will use left word metrics.

Define $\zeta\colon \Z\wr_S G\to SV_G$ as follows. Embed $\Z$ into $V$ by sending $1$ to an element of $V$ acting on the dyadic brick $B(0)$ via $0\kappa\mapsto 00\kappa$. Extend this to an embedding $\bigoplus\limits_S \Z \to \bigoplus\limits_S V$. The coordinate-wise action of $\bigoplus\limits_S V$ on $C^S$ provides an embedding into $SV$. Finally, $G$ embeds in $SV_G$ via $\iota_\varnothing \colon g\mapsto [\cdot,\id,g,\cdot]$. The images of these embeddings generate a copy of $\Z\wr_S G$ in $SV_G$, and so we have our monomorphism $\zeta\colon \Z\wr_S G\to SV_G$. By construction, $\rho_{\kappa_0}\circ\zeta$ is the identity on $\Z\wr_S G$. All that remains to see that $\rho_{\kappa_0}$ is a quasi-retraction is to check that it and $\zeta$ are coarse Lipschitz. This holds for $\zeta$, which is a homomorphism, and for $\rho_{\kappa_0}$ it holds since for all $h\in SV_G$ and all $a\in A$ there exists $b\in B$ such that $\rho_{\kappa_0}(ah) = b\rho_{\kappa_0}(h)$.
\end{proof}

We may as well also handle the homological case. Say $G\curvearrowright S$ is of \emph{type~$[\HA_n]$} if $G$ is of type $\FP_n$, $\Stab_G(T)$ is of type $\FP_{n-|T|}$ for all finite $T\subseteq S$, and the diagonal action of $G$ on $S^n$ has finitely many orbits. Then $\Z\wr_S G$ is of type $\FP_n$ if and only if $G\curvearrowright S$ is of type~$[\HA_n]$ \cite{bartholdi15}.

\begin{corollary}\label{cor:qr_fin_props}
For any $n\ge 1$, if $SV_G$ is of type $\F_n$, then $G\curvearrowright S$ is of type~$[\A_n]$. In particular if $SV_G$ is finitely presented then $G\curvearrowright S$ is of type~$[\A_2]$. Similarly for type $\FP_n$ and type~$[\HA_n]$.
\end{corollary}

\begin{proof}
If $SV_G$ is of type $\F_n$, then since $\Z\wr_S G$ is a quasi-retract of $SV_G$ (Proposition~\ref{prop:qr}), it is also of type $\F_n$ \cite{alonso94}, and hence $G\curvearrowright S$ is of type~$[\A_n]$ \cite{cornulier06,bartholdi15}. The proof for type $\FP_n$ is identical.
\end{proof}

Now we have both directions of the finite presentability and type~$\F_\infty$ statements of Theorem~\ref{thrm:fin_props}.

\begin{proof}[Proof of Theorem~\ref{thrm:fin_props} parts~(ii) and~(iii)]
Part (ii) follows by combining Corollaries~\ref{cor:yes_fp} and~\ref{cor:qr_fin_props}, and part (iii) follows by combining Proposition~\ref{prop:yes_Finfty} with Corollary~\ref{cor:qr_fin_props}.
\end{proof}

Similarly, we obtain a converse to Corollary~\ref{cor:yes_FP2}.

\begin{corollary}\label{cor:iff_FP2}
$SV_G$ is of type~$\FP_2$ if and only if $G\curvearrowright S$ is of type~$[\HA_2]$.\qed
\end{corollary}

We know that the converse of Corollary~\ref{cor:qr_fin_props} holds for $n=1,2,\infty$, so a natural question is whether it holds when $3\le n<\infty$, i.e., whether in these cases $G\curvearrowright S$ being of type~$[\A_n]$ is sufficient for $SV_G$ to be of type $\F_n$. Indeed, this is even unknown in the faithful case, see Conjecture~H of \cite{belk22}, where this direction is still open. Similarly one would expect $G\curvearrowright S$ being of type~$[\HA_n]$ to be sufficient for $SV_G$ to be of type $\FP_n$.

\section{Embedding results}\label{sec:embed}

In this section we connect abstract twisted Brin--Thompson groups to the Boone--Higman conjecture.

\begin{theorem}[Theorem~\ref{thrm:main:embed}]
\label{thrm:embedding}
Let $\Gamma$ be a finitely presented simple group. Then there exists a group $G$ with a type~$[\A_2]$ action on a set $S$ such that $\Gamma$ sharply embeds in $(G,\ker(G\curvearrowright S))$. Hence $\Gamma$ sharply embeds in the finitely presented (relatively simple) abstract twisted Brin--Thompson group $SV_G$.
\end{theorem}

First let us record the following easy lemma:

\begin{lemma}\label{lem:embed_in_tbt}
Let $G\curvearrowright S$ be a group action with kernel $K$. Then the normal pair $(G,K)$ sharply embeds in $SV_G$.
\end{lemma}

\begin{proof}
Send $g\in G$ to $[\cdot,\id,g,\cdot]$, where $\cdot$ is the trivial multicolored tree. To see that this gives the desired embedding, we just have to prove that if $[\cdot,\id,g,\cdot] \in SK_G$ then $g\in K$. Indeed, if $g\not\in K$ then the action of $g$ on $C^S$ by permuting coordinates is non-trivial.
\end{proof}

\begin{proof}[Proof of Theorem~\ref{thrm:embedding}]
We follow the strategy from \cite{BFFHZ}. Let $G=\Aut_\Gamma(\Gamma*F_n)$ be the group of automorphisms of the free product $\Gamma*F_n$, for $F_n$ the free group of rank $n\ge 2$, that restrict to the identity on the $\Gamma$ factor. Let $S=\Hom_\Gamma(\Gamma*F_n,\Gamma)$ be the set of homomorphisms $\Gamma*F_n\to\Gamma$ that restrict to the identity on the $\Gamma$ factor. We have an action of $G$ on $S$ by precomposition. Since $\Gamma$ is simple, this action is highly transitive \cite[Proposition~2.3]{BFFHZ}, hence $G\curvearrowright S\times S$ has finitely many orbits. Since $\Gamma$ is finitely generated, each $\Stab_G(s)$ for $s\in S$ is finitely generated by \cite[Proposition~2.6]{BFFHZ}. Since $\Gamma$ is finitely presented and has trivial center, $G$ is finitely presented by \cite[Proposition~1.1]{BFFHZ}. Thus $G\curvearrowright S$ is of type~$[\A_2]$. The group $\Gamma$ embeds into $G$ as the $\Gamma$-translations of some fixed generator of $F_n$. The only remaining thing we have to prove is that the image of this embedding intersects the kernel of the action trivially, i.e., that the restriction of $G\curvearrowright S$ to $\Gamma$ is faithful. In fact it is free, as we will show. Let $\phi\in S$ send our fixed generator of $F_n$ to $\gamma_0$. Now for any $\gamma\in\Gamma$, viewed in $G$ as above, we have that $\phi\circ\gamma$ sends our fixed generator of $F_n$ to $\gamma_0\gamma$. Thus $\phi\circ\gamma$ cannot equal $\phi$ unless $\gamma=1$, so we conclude that the action of $\Gamma$ on $S$ is free. The last statement is immediate from Theorem~\ref{thrm:fin_props}(ii) and Lemma~\ref{lem:embed_in_tbt}.
\end{proof}

\begin{remark}
The action $G \curvearrowright S$ constructed in the proof above is of type $[\A_2]$ but not of type $[\A_\infty]$ in general, for example when $\Gamma$ is a (finitely presented, simple) Burger--Mozes group \cite[Corollary~1.6]{FFKLZ}.
\end{remark}

We should emphasize that one advantage of allowing non-faithful actions is that, \emph{a priori}, the simple quotient obtained by taking $SV_G$ modulo its largest subgroup is not necessarily finitely presented, and so even though sharpness of the embedding induces an embedding into the quotient, this quotient is not necessarily useful for Boone--Higman-related purposes. (That being said, we do not know an explicit example where it fails, see Question~\ref{quest:not_fp}.) This is in contrast to other examples of relatively simple groups, for instance in the braided Thompson group or labeled Thompson groups the quotient by the largest subgroup is $V$, which is finitely presented. Also, for the relatively simple lamplighters in Lemma~\ref{lem:lamp}, if the wreath product is finitely presented then so is the quotient by the largest normal subgroup, since the quotient splits.

\medskip

It is not difficult to see that the converse of the (relative) (permutational) Boone--Higman conjecture is true, in fact the relative Boone--Higman conjecture is a natural strengthening of a characterization of groups with solvable word problem that follows from the Boone--Higman--Thompson theorem, and that we present in the appendix (Proposition \ref{prop:strong_bht}). Note that a finitely presented relatively simple group itself may have unsolvable word problem, for example if $|S| = 1$ and $G$ has unsolvable word problem, then $SV_G$ has unsolvable word problem (but any group sharply embedding in it has solvable word problem). As a remark, the braided Thompson group $bV$ is an example of a finitely presented relatively simple group with solvable word problem for which the Boone--Higman conjecture (and even the relative Boone--Higman conjecture) remains open.

\medskip

Let us record a few observations about the relative permutational Boone--Higman (relPBH) conjecture, and actions of type~$[\A_2]$.

\begin{observation}\label{obs:trans}
If a group admits an action of type~$[\A_n]$ then it admits a transitive action of type~$[\A_n]$.
\end{observation}

\begin{proof}
The restriction of the action to a single orbit still has finitely many orbits of $n$-tuples.
\end{proof}

Interestingly, trying to use the analogous proof for type~$(\A_n)$ does not always work, since the restriction of a faithful action to a single orbit need not be faithful.

Recall that the \emph{bi-index} of a subgroup $H$ of a group $G$ is the cardinality of the set of double cosets $HgH$ ($g\in G$). The following gives a purely group theoretic equivalent condition to satisfying the relative permutational Boone--Higman conjecture.

\begin{lemma}\label{lem:relpbh_criterion}
The following are equivalent for a group $\Gamma$:
\begin{enumerate}
    \item $\Gamma$ satisfies the relative permutational Boone--Higman conjecture, i.e., there is a group $G$ with a type~$[\A_2]$ action on a set $S$ such that $\Gamma$ sharply embeds in the normal pair $(G,\ker(G\curvearrowright S))$.
    \item The same as item (i) but moreover the action is transitive.
    \item There exists a finitely presented group $G$ and a finitely generated subgroup $H\le G$ with finite bi-index, such that $\Gamma$ embeds in $G$ with every non-trivial element of $\Gamma$ conjugate in $G$ to an element of $G\setminus H$.
\end{enumerate}
\end{lemma}

\begin{proof}
First assume (i) holds, so $G\curvearrowright S$ is of type~$[\A_2]$ and $\Gamma$ sharply embeds in $(G,\ker(G\curvearrowright S))$, say via $\iota\colon\Gamma\to G$. Note that by Observation~\ref{obs:trans}, we could assume the action is transitive by just restricting to a single orbit, but then the embedding of $\Gamma$ might no longer be sharp, so more work is needed. Let $S_1,\dots,S_n\subseteq S$ be the $G$-orbits of $G\curvearrowright S$. Consider the action of the finitely presented group $G^n$ on $S_1\times\cdots\times S_n$ via $(g_1,\dots,g_n).(s_1,\dots,s_n)\coloneqq (g_1.s_1,\dots,g_n.s_n)$, so this action is transitive. Since $G\curvearrowright S$ has finitely many orbits of pairs, the same is true of this action of $G^n$ on $S_1\times\cdots\times S_n$. Since point stabilizers for $G\curvearrowright S$ are finitely generated, the same is true for the new action. Finally, embed $\Gamma$ into $G^n$ diagonally via $\iota$ in each factor, and notice that since each non-trivial element of $\Gamma$ must act non-trivially on some $S_i$, the embedding is sharp relative the kernel of the action of $G^n$. This shows (ii).

Now assume (ii) holds, so $G\curvearrowright S$ is of type~$[\A_2]$ and transitive, and $\Gamma$ sharply embeds in $(G,\ker(G\curvearrowright S))$. Choose $s\in S$ and let $H=\Stab_G(s)$, so $H$ is finitely generated and has finite bi-index. Since every non-trivial element $g$ of $\Gamma$ acts non-trivially on $S$, $g$ can be conjugated out of $H$, which confirms (iii).

Finally, given the assumptions of (iii), let $S$ be the set of cosets $G/H$ with the $G$-action by left translation, so the assumptions ensure this action is of type~$[\A_2]$. The kernel of this action is the intersection of all conjugates of $H$, so $\Gamma$ intersects this trivially. Thus, $\Gamma$ sharply embeds in $(G,\ker(G\curvearrowright S))$.
\end{proof}

In words, if one wishes to prove (relPBH) for a given group $\Gamma$, one should hunt for a finitely presented group $G$ and a finitely generated subgroup $H\le G$ of finite bi-index, such that $\Gamma$ admits an ``$H$-escapable'' embedding into $G$. (Here $H$-escapable means any non-trivial element of $\Gamma$ can be conjugated out of $H$.)

\begin{remark}\label{rmk:bht_and_he}
The Boone--Higman conjecture ties together two classical results: the Boone--Higman--Thompson theorem \cite{boone74,thompson80}, and Higman's embedding theorem \cite{higman61}. Restricting to always talking about finitely generated groups for conciseness, the former says that every group with solvable word problem embeds in a recursively presented simple group, and the latter says that every recursively presented group embeds in a finitely presented group. Thus, one can work out that the Boone--Higman conjecture is equivalent to a version of the Boone--Higman--Thompson theorem that preserves finite presentability, and also to a version of the Higman embedding theorem that preserves simplicity (here one needs the fact, due to Clapham \cite{clapham67}, that every group with solvable word problem embeds in a finitely presented group with solvable word problem). In fact, twisted Brin--Thompson groups provide a stronger version of the Boone--Higman--Thompson theorem, since a finitely generated group $G$ has solvable word problem if and only if the finitely generated simple group $GV_G$ is recursively presented \cite[Corollary~4.14]{bbmz_survey}, so the Boone--Higman conjecture is equivalent to always being able to embed such a $GV_G$ into a finitely presented simple group.
\end{remark}

Let us conclude this section by collecting some of the relevant questions that have come up.

\begin{question}\label{quest:tBT_to_action}
For $SV_G$ a finitely presented abstract twisted Brin--Thompson group, does the normal pair $(SV_G,SK_G)$ sharply embed in a normal pair of the form $(G',\ker(G'\curvearrowright S'))$ for some $G'\curvearrowright S'$ of type $[\A_2]$?
\end{question}

\begin{question}\label{quest:not_fp}
Does there exist an example of a group action $G\curvearrowright S$ that is of type $[\A_2]$ such that $G/\ker(G\curvearrowright S)$ is not finitely presented? (This is equivalent to the induced action of this quotient on $S$ being not of type~$(\A_2)$.)
\end{question}

\begin{question}\label{quest:relBH_to_BH}
For $G\curvearrowright S$ an action of type $[\A_2]$ with kernel $K$, does there always exist an action $G'\curvearrowright S'$ of type $[\A_2]$ with kernel $K'$ finitely normally generated, such that $(G,K)$ sharply embeds in $(G',K')$? Does every finitely presented relatively simple $(G,N)$ sharply embed in a finitely presented relatively simple $(G',N')$ such that $N'$ is finitely normally generated in $G'$?
\end{question}

Note that a ``yes'' answer to the first part of Question~\ref{quest:relBH_to_BH} would tell us that (relPBH) implies (PBH), and the second part would tell us that (relBH) implies (BH). By Theorem~\ref{thrm:main:embed} the first part would imply that (BH), (relPBH), and (PBH) are all equivalent to each other, and the second part would imply that (BH), (relPBH), and (relBH) are all equivalent to each other.

\section{Geometric properties}\label{sec:geom}

In this section we investigate some geometric properties of abstract twisted Brin--Thompson groups, analogous to known results for faithful twisted Brin--Thompson groups.

First let us quickly record the following geometric fact, which is essentially immediate from the proof of Proposition~\ref{prop:qr}. The faithful case of this was one of the main results of \cite{belk22}.

\begin{corollary}\label{cor:QI_embed}
Let $G\curvearrowright S$ be an action of type $[\A_1]$, so $G$ and $SV_G$ are finitely generated. Then $G$ is a quasi-retract of $SV_G$, in particular it is quasi-isometrically embedded.
\end{corollary}

\begin{proof}
Proposition~\ref{prop:qr} gives a quasi-retraction $SV_G \to \Z \wr_S G$, which in turn admits an honest retraction onto $G$. Quasi-retracts are quasi-isometrically embedded.
\end{proof}

The main goal of this section is to understand all the ways in which abstract twisted Brin--Thompson groups can(not) act on hyperbolic spaces. First we will prove that the groups are always uniformly perfect, then we will analyze hyperbolic actions (which in particular will use uniform perfectness), and finally we will use this to analyze actions on CAT(0) cube complexes.

\subsection{Uniform perfectness}\label{ssec:unif_perf}

Recall that a group is \emph{$n$-uniformly perfect} if every element is the product of $n$ commutators. This property is quite common among homeomorphism groups \cite{gal:gismatullin}.

\begin{theorem}\label{thrm:up}
Let $G$ be a group acting on a non-empty set $S$. Then $SV_G$ is $5$-uniformly perfect.
\end{theorem}

In case the action is faithful, \cite[Theorem~5.1]{gal:gismatullin} applies and shows that $SV_G$ is $3$-uniformly perfect - see the proof of \cite[Proposition~5.23]{NL}. In the case of $V(G)$, Theorem~\ref{thrm:up} was proved in \cite[Theorem~2.4]{wu25}. Using the faithful case as input, a very similar argument applies.

Let us start with an elementary computation that we will use once again in the next subsection:
\begin{equation}
\label{eq:conjugacy}
    [T, \id, (g_1, \ldots, g_n), T]^{[T, \sigma, (1, \ldots, 1), T]} = [T, \id, (g_{\sigma(1)}, \ldots, g_{\sigma(n)}), T].
\end{equation}

\begin{proof}[Proof of Theorem~\ref{thrm:up}]
    By the discussion above, $SV_{G/K}$ is $3$-uniformly perfect. Therefore the set of products of $3$ commutators in $SV_G$ maps surjectively onto $SV_{G/K}$, which reduces the statement to proving that every element in the kernel $SK_G$ is a product of $2$ commutators in $SV_G$.

    For an element of $SK_G$ we have the identity
    \[[T, \id, (k_1, k_2, \ldots, k_n), T] = [T, \id, (k_1, 1, \ldots, 1), T] [T, \id, (1, k_2, \ldots, k_n), T].\]
    By Equation~\eqref{eq:conjugacy}, both elements on the right hand side are conjugate to an element of the form
    \[v = [T, \id, (1, k_2, \ldots, k_n), T],\]
    hence it suffices to show that such elements are commutators.

    Let $T'$ be obtained from $T$ by splitting its leftmost brick $(n-1)$ times according to some fixed color $s$. Then we have
    \[v = [T', \id, (1, \ldots, 1, k_2, \ldots, k_n), T'],\]
    where the number of $1$s is $n$.
    Let $T''$ be the tree obtained from $T$ by splitting the brick corresponding to each leaf except the first one once according to the same color $s$. Then we also have
    \[v = [T'', \id, (1, k_2, k_2, \ldots, k_n, k_n), T''].\]
    Note that $T'$ and $T''$ each have $2n-1$ leaves.
    
    Now let $\alpha$ be the permutation of $2n-1$ letters such that applying it to the entries of $(1, \ldots, 1, k_2, \ldots, k_n)$ yields $(1, k_2, 1, \ldots, 1, k_n, 1)$. Let $a = [T'', \alpha, (1, \ldots, 1), T']$, then
    \[vav^{-1}a^{-1} = [T'', \id, (1,1, k_2, 1, k_3, \cdots, 1,k_n), T''].\]
    Similarly, let $\beta$ be the permutation sending $(1,1, k_2, 1, k_3, \ldots, 1,k_n)$ to $(1, \ldots, 1, k_2, \ldots, k_n)$, and let $b = [T', \beta, (1, \ldots, 1), T'']$. Then
    \[b(vav^{-1}a^{-1})b^{-1} = [T', \id, (1, \ldots, 1, k_2, \ldots, k_n), T'] = v,\]
    which concludes the proof.
\end{proof}

\subsection{Actions on hyperbolic spaces}\label{ssec:hyp}

Next, we prove that all abstract twisted Brin--Thompson groups $SV_G$ have property NL, as defined in \cite{NL}. This extends the faithful case \cite[Proposition~5.23]{NL}. To the best of our knowledge, this is new also for $V(G)$.

We start with preliminary definitions that are also needed for the statement. Let $X$ be a (Gromov-)hyperbolic space, let $\partial X$ denote its boundary, and consider the compact space $X \cup \partial X$ endowed with its visual topology. If $g$ is an isometry of $X$, we say that $g$ is
\begin{itemize}
    \item \emph{elliptic} if it has bounded orbits;
    \item \emph{hyperbolic} if $\{g^n x\}_{n \in \mathbb{Z}}$ defines a quasi-isometric embedding of $\mathbb{Z}$ into $X$, for some (equivalently every) point $x \in X$;
    \item \emph{parabolic} otherwise.
\end{itemize}

Given a group $G$ acting on $X$ by isometries, the \emph{limit set} is
\[\Lambda(G) \coloneqq \overline{Gx} \cap \partial X,\]
where $x \in X$ is any point; the definition is independent of this choice. Now $g$ is elliptic if $\Lambda(\langle g \rangle) = \emptyset$, hyperbolic if $\Lambda(\langle g \rangle) = \{g^{\pm \infty}\}$ consists of two points, and parabolic if $\Lambda(\langle g \rangle) = \{g^\infty\}$ consists of a single point. These points at infinity are called \emph{limit points} of $g$.

More generally, group actions on hyperbolic spaces can be classified in terms of their limit sets. We say that the action of $G$ on $X$ is:
\begin{itemize}
    \item \emph{elliptic} if $\Lambda(G) = \emptyset$, equivalently if $G$ has bounded orbits;
    \item \emph{horocyclic} (or \emph{parabolic}) if $|\Lambda(G)| = 1$, equivalently if $G$ fixes a unique point on $\partial X$ and contains no loxodromic element;
    \item \emph{lineal} if $|\Lambda(G)| = 2$, equivalently if $G$ contains a loxodromic element and any two loxodromic elements have the same limit points; in this case we moreover say that the action is \emph{oriented} if the two points in the limit set are fixed by all of $G$, and \emph{non-oriented} otherwise;
    \item \emph{focal} (or \emph{quasi-parabolic}) if $|\Lambda(G)| = \infty$ and $G$ has a global fixed point on $\partial X$, equivalently if any two loxodromic elements share a common limit point, but there exist pairs of loxodromic elements not sharing both limit points;
    \item \emph{general type} if $|\Lambda(G)| = \infty$ and $G$ has no global fixed point on $\partial X$, equivalently if $G$ contains two loxodromic elements with disjoint sets of limit points.
\end{itemize}

Any action is of one of these forms, see \cite[Section 8.2]{gromov:hyp}. Every group admits an elliptic action on a point, and every countably infinite group admits a parabolic action on a hyperbolic space (for finitely generated groups, on combinatorial horoballs on Cayley graphs \cite{groves:manning}, and for infinitely generated groups, on trees \cite[Chapter 6]{serre:trees}). All other actions involve loxodromic elements, which motivates the following definition.

\begin{definition}
    A group has \emph{Property NL}, short for \emph{No Loxodromics}, if every action on a hyperbolic space is either elliptic or parabolic.
\end{definition}

Many interesting groups have property NL, such high rank lattices \cite{NL:lattices1, NL:lattices2} and many Thompson-like groups \cite{NL:V, NL} (although $F$-like groups do not \cite{F:hyperbolic}, nor do braided Thompson groups \cite{qm:bV}). In particular $SV_G$ has property NL when $G \curvearrowright S$ is faithful \cite[Proposition~5.23]{NL}. The faithfulness is important in that proof, since it is an application of a general criterion that gives property NL for a class of groups of homeomorphisms of compact Hausdorff spaces \cite[Theorem~5.1]{NL}. Nevertheless, we now prove:

\begin{theorem}
\label{thrm:NL}
    Let $G$ be a group acting on a non-empty set $S$. Then $SV_G$ has property NL.
\end{theorem}

Before starting the proof, let us recall that a map $\varphi \colon G \to \mathbb{R}$ is a \emph{quasimorphism} if its defect
\[D(\varphi) \coloneqq \sup\limits_{g, h \in G} |\varphi(g) + \varphi(h) - \varphi(gh)|\]
is finite. Every action on a hyperbolic space fixing a point at infinity defines a quasimorphism, which is unbounded precisely on loxodromic elements \cite[Section 4.1]{manning}. It is easy to see that quasimorphisms are uniformly bounded on commutators, hence uniformly bounded on products of at most $n$ commutators, for a fixed $n$. In particular, on a uniformly perfect group, every quasimorphism is bounded. We thus have the following corollary of Theorem~\ref{thrm:up}:

\begin{corollary}
\label{cor:NL:step}
    Let $G$ be a group acting on a non-empty set $S$. Then every action of $SV_G$ on a hyperbolic space is elliptic, horocyclic, or of general type.
\end{corollary}

\begin{proof}
    Because $SV_G$ is relatively simple (Theorem~\ref{thrm:rel_simple}) with infinite simple quotient, it has no subgroup of index $2$. This implies that every lineal action is oriented. Therefore every action that is not elliptic, horocyclic, or of general type, fixes a point at infinity and has a loxodromic element. By \cite[Section 4.1]{manning}, this would produce an unbounded quasimorphism on $SV_G$. But by Theorem~\ref{thrm:up}, $SV_G$ is $5$-uniformly perfect, hence every quasimorphism is bounded.
\end{proof}

We are ready to complete the proof.

\begin{proof}[Proof of Theorem~\ref{thrm:NL}]
    By Corollary~\ref{cor:NL:step}, it remains to exclude actions of general type. So suppose that $SV_G \curvearrowright X$ is a general type action on a hyperbolic space.

    Consider the action of $SK_G$ on $X$, and suppose first that it is horocyclic, lineal, or focal. Then $SK_G$ has a unique finite orbit on $\partial X$, and by normality this must be also preserved by $SV_G$, contradicting that the action of $SV_G$ is of general type.

    Suippose next that the action of $SK_G$ is elliptic. Then by \cite[Lemma~4.10]{NL} this induces a general type action of the quotient $SV_{G/K}$ on some other hyperbolic space, contradicting the fact that $SV_{G/K}$ has property NL \cite[Proposition~5.23]{NL}.

    Finally, suppose that the action of $SK_G$ is of general type, that is, there exist two loxodromic elements with disjoint set of limit points. For every multicolored tree $T$, denote
    \begin{equation}\label{eq:KT}
        K_T \coloneqq \{[T, \id, (k_1, \ldots, k_n), T] \mid k_1, \ldots, k_n \in K \} \cong K^n,
    \end{equation}
    and note that $SK_G$ is the directed union of the groups $K_T$. Choosing $T$ large enough so that $K_T$ contains the two loxodromic elements above, we see that the action of $K_T$ is also of general type. But \cite[Lemma~4.4]{NL} implies that not all direct factors of $K_T$ can act elliptically, and then \cite[Lemma~4.3]{NL} implies that the action of some direct factor must be of general type. But all direct factors isomorphic to $K$ are conjugate by Equation~\eqref{eq:conjugacy}, so the action of every factor is of general type.

    This implies that there exists a loxodromic element $g$ (in one of the $K$-factors) such that the action of its centralizer $C$ is of general type (since it contains one of the other $K$-factors). But if $C$ centralizes $g$ it must stabilize $\{g^{\pm \infty} \}$, which reaches a final contradiction and concludes the proof.
\end{proof}

\subsection{Actions on finite-dimensional CAT(0) cube complexes}

We say that a group has \emph{property FW$_\infty$} if every action on a finite-dimensional CAT(0) cube complex has a global fixed point. This was introduced by Barnhill and Chatterji in \cite{guido} with the name ``property FW'', however since then the name has been used to encompass also actions on infinite-dimensional CAT(0) cube complexes \cite{cornulier:FW}, so we use the ``property FW$_\infty$'' convention of \cite{NL:V}. Thanks to a criterion of Genevois \cite[Theorem 5.1]{NL:V}, our results on property NL immediately imply:

\begin{corollary}
\label{cor:FWinfty}
    Let $G$ be a group acting on a non-empty set $S$. Then $SV_G$ has property FW$_\infty$.
\end{corollary}

\begin{proof}
    This follows directly from \cite[Theorem 5.1]{NL:V}, together with the fact that $SV_G$ has no proper finite index subgroups, by Theorem~\ref{thrm:rel_simple}; see \cite[Proposition 6.2]{NL}.
\end{proof}

\begin{remark}
This invites the question of whether $SV_G$ can act on an infinite dimensional CAT(0) cube complex without a global fixed point, see \cite{cornulier:FW} for several equivalent definitions. When $|S|=1$, so $SV_G=V(G)$, the Stein complex $SX_G=X(G)$ (Definition~\ref{def:stein}) admits a coarser, cubical structure, with respect to which it is a CAT(0) cube complex. This is often called the Stein--Farley complex (following Farley's work in \cite{farley03}), and the action of $V(G)$ on $X(G)$ has unbounded orbits (see the proof of \cite[Corollary~3.17]{wu25}). When $G$ is finite the action is even proper, revealing that $V(G)$ has the Haagerup property for finite $G$.

When $|S|\ge 2$ the picture is different. It turns out $2V$, hence every $SV_G$ for $|S|\ge 2$, contains distorted elements \cite{Callard24}. By \cite{Haglund23}, this means $SV_G$ cannot act properly on a CAT(0) cube complex when $|S|\ge 2$. To the best of our knowledge it is open whether $2V$ (more generally $SV_G$ for $|S|\ge 2$) can act with unbounded orbits on an infinite dimensional CAT(0) cube complex. Relatedly, as far as we know, the Haagerup and Kazhdan properties are open for $2V$.
\end{remark}

\section{Homological properties}\label{sec:homological}

In this section we prove that all abstract twisted Brin--Thompson groups are boundedly acyclic and $\ell^2$-invisible. The former generalizes the same result in the faithful case, proved by Wu, Zhao, Zhou, and the second author \cite{wu25}. The latter is new even in the faithful case, and follows from a new criterion for $\ell^2$-invisibility that is interesting in its own right (Theorem~\ref{thrm:general_l2_criterion}).

First let us pin down a definition that will be useful for our proofs of both bounded acyclicity and $\ell^2$-invisibility.

\begin{definition}[Full deferment]
\label{def:full:deferment}
Let $G$ be a group acting on a non-empty set $S$. Let $U$ be an open subset of $C^S$. Define the \emph{full deferment} $D_U(SV_G)$ of $SV_G$ to $U$ to be
\[
D_U(SV_G)\coloneq \{h\in SV_G\mid h(\kappa)=\kappa \text{ and } \gtwist_\kappa(h)=1 \text{ for all } \kappa\in C^S\setminus U\}\text{.}
\]
\end{definition}

Recall that in Section~\ref{sec:simple} we dealt with the deferment $D_B(g)$ of a single element $g$ of $G$ to a dyadic brick $B$, whereas here the full deferment is of the entire group $SV_G$ to an open subset $U$, roughly speaking.

\subsection{Bounded acyclicity}\label{ssec:bounded}

Bounded cohomology is a functional-analytic analog of cohomology with many applications in geometric topology \cite{gromov:bc}, rigidity theory \cite{burmon, monshal} and one-dimensional dynamics \cite{ghys}. A group $G$ is called \emph{boundedly acyclic} if its bounded cohomology $H_b^i(G,\mathbb{R})$ is trivial for all $i\geq 1$. We refer the reader to \cite[Section~6]{ccc} for a discussion of some useful consequences of vanishing results for bounded cohomology; let us only mention that vanishing in degree $2$ is closely related to the rigidity of quasimorphisms that we obtained through uniform perfectness (Theorem~\ref{thrm:up}) and applied to prove property NL (Theorem~\ref{thrm:NL}).

The past few years have seen many new bounded acyclicity results, especially for groups of dynamical origin and Thompson-like groups \cite{binate, monod22, monod23, ccc, ffmn}. The most relevant result for us is the bounded acyclicity of both labeled Thompson groups and faithful twisted Brin--Thompson groups \cite{wu25}. Here we generalize both results by showing the bounded acyclicity of all abstract twisted Brin--Thompson groups.

\begin{definition}\label{generic}
Let $X$ be a set.  We call a binary relation $\perp$ on $X$ \emph{generic} if  for every given finite set $Y \subseteq X$,  there is an element $x \in X$ such that $y \perp x$ for all $y \in Y$. 
\end{definition}

A generic relation on $X$ gives rise to a semi-simplicial set $X_{\bullet}^{\perp}$ in the following way:  we define $X_{n}^{\perp}$ to be the set of all $(n+1)$-tuples $\left(x_{0},  \ldots,  x_{n}\right) \in X^{n+1}$ for which $x_{i} \perp x_{j}$ holds for all $0 \leq i<j \leq n$.  The face maps $\partial_{n}\colon  X_{n}^{\perp} \rightarrow X_{n-1}^{\perp}$ are the usual simplex face maps,  i.e.,
$$
\partial_{n}\left(x_{0},  \ldots,  x_{n}\right)=\sum_{i=0}^{n}(-1)^{i}\left(x_{0},  \ldots,  \widehat{x}_{i},  \ldots,  x_{n}\right)
$$
where $\widehat{x}_{i}$ means that $x_{i}$ is omitted. 

The following theorem is a direct consequence of \cite[Proposition~3.2]{monod23} and \cite[Theorem~3.3 and Remark~3.4]{monod23}.

\begin{cit}[{\cite[Corollary~4.4]{wu25}}]
\label{cit:BA-gene}
    Let $\Gamma$ be a group acting on a set $X$ and $\perp$ a generic relation on $X$ preserved by $\Gamma$.  If for every $n \geq 0$,  the action of $\Gamma$ on $X^\perp_n$ is transitive  and the stabilizers are boundedly acyclic, then $\Gamma$ is boundedly acyclic.
\end{cit}

The following criterion will be useful for showing bounded acyclicity of stabilizers.

\begin{cit}\label{cit:coame-to-wreath} \cite[Corollary~5]{monod22}
Let $\Gamma$ be a group and $\Gamma_{0}<\Gamma$ a subgroup with the following two properties: 
\begin{enumerate}
    \item Every finite subset of $\Gamma$ is contained in some $\Gamma$-conjugate of $\Gamma_{0}$, 
    \item $\Gamma$ contains an element $g$ such that the conjugates of $\Gamma_{0}$ by $g^{p}$ and by $g^{q}$ commute for all $p \neq q$ in $\mathbb{Z}$. 
\end{enumerate}
Then $\Gamma$ is boundedly acyclic. 
\end{cit}

We now proceed to the proof of bounded acyclicity for abstract twisted Brin--Thompson groups. Let $\kappa_0 \in C^S$ be the point defined by $\kappa_0(s)=\overline{0}$ for all $s\in S$.

\begin{definition}
Given $h_1,h_2\in SV_G$, call $h_1$ and $h_2$ \emph{equivalent} if there is a dyadic brick $B$ containing $\kappa_0$ such that $h_1(\kappa) = h_2(\kappa) \text{ and } \gtwist_\kappa(h_1) = \gtwist_\kappa(h_2) \text{ for all } \kappa \in B$. A \emph{labeled germ} is an equivalence class of $SV_G$ under this relation. Let us denote the equivalence class of $h\in SV_G$ by $[h]$ and the set of all labeled germs by $\mathcal{G}$.
\end{definition}

Consider the action of $SV_G$ on $\mathcal{G}$ by $f.[h] = [fh]$. This is well defined since $\gtwist_\kappa(fh) = \gtwist_{h(\kappa)}(f)\gtwist_\kappa(h)$ for all $f$, $h$, and $\kappa$, by Equation~\eqref{eq:twist_two}.

\begin{definition}
    Given two labeled germs $[h_1],[h_2]$, we declare $[h_1] \perp [h_2]$ if $h_1(\kappa_0) \ne h_2(\kappa_0)$.
\end{definition}

Since the $SV_G$-orbit of $\kappa_0$ is infinite, $\perp$ defines a generic relation on $\mathcal{G}$. Note also that $SV_G$ preserves the relation. Let $\mathcal{G}_n^\perp$ denote the set of generic $(n+1)$-tuples of $\mathcal{G}$.

\begin{lemma}\label{lemma:svg-action-trans}
The group $SV_G$ acts transitively on $\mathcal{G}_n^\perp$ for each $n\geq 0$.
\end{lemma}

\begin{proof}
The proof is almost identical to \cite[Lemma~5.13]{wu25}, so we will just sketch it. Let $([h_0],\dots,[h_n])$ and $([h_0'],\dots,[h_n'])$ be generic $(n+1)$-tuples, so for all $i\ne j$ we have $h_i(\kappa_0)\ne h_j(\kappa_0)$ and $h_i'(\kappa_0)\ne h_j'(\kappa_0)$. We can choose a dyadic brick $B$ containing $\kappa_0$ such that the $h_i$-translates of $B$ are pairwise disjoint, as are the $h_i'$-translates. Now it is easy to construct an element $f$ of $SV_G$ taking $h_i(B)$ to $h_i'(B)$ for each $i$, such that $[fh_i]=[h_i']$ for each $i$.
\end{proof}

\begin{definition}[Labeled support]
Given an element of $f\in SV_G$, we define its \emph{labeled support} $\LSupp(f)$ to be the closure of the set $\{ \kappa \in C^S \mid f(\kappa) \neq  \kappa \text{ or } \gtwist_\kappa(f) \neq 1 \}$.
\end{definition}

In particular the labeled support of an element contains its support, in the sense of the action on $C^S$. For a dyadic brick $B \subset C^S$, we have that
\[
\LSupp(f) \subseteq B \quad \Leftrightarrow \quad f \in D_B(SV_G),
\]
where $D_B(SV_G)$ is the full deferment (Definition~\ref{def:full:deferment}). Note that two elements of $SV_G$ with disjoint labeled support commute.

\begin{theorem}\label{thrm:SV_G_bdd_acyc}
Let $G$ be a group acting on a non-empty set $S$. Then the abstract twisted Brin--Thompson group $SV_G$ is boundedly acyclic. 
\end{theorem}

\begin{proof}
Consider the action of $SV_G$ on the set $\mathcal{G}$. By Citation~\ref{cit:BA-gene} and Lemma~\ref{lemma:svg-action-trans}, to prove that $SV_G$ is boundedly acyclic it suffices now to prove that the stabilizers for the action on $\mathcal{G}_n^{\perp}$ are boundedly acyclic for all $n\geq 0$.

Let $\vec{h}=([h_0], \ldots, [h_n]) \in \mathcal{G}_n^\perp$, and we must show that $\Stab_{SV_G}(\vec{h})$ is boundedly acyclic. Fix $s_0\in S$, let $\psi_0\colon S\to \{0,1\}^*$ be the map that sends $s_0$ to $00$ and all other $s$ to the empty word, and let $B(\psi_0)$ be the corresponding dyadic brick. Since the action is transitive, we can assume that $h_i(\kappa_0) \in B(\psi_0)$ for all $0\le i\le n$. Writing $e$ for the identity element of $SV_G$, note that $f$ fixes $[e]$ if and only if $\kappa_0 \not \in \LSupp(f)$. Thus
\[
\Stab_{SV_G}(\vec{h}) = \bigcap_{i=0}^n h_i \Stab_{SV_G}([e]) h_i^{-1} = \{ f\in SV_G\mid  h_i(\kappa_0)  \not \in  \LSupp(f), 0\leq i\leq n\} \text{.}
\]

Let $\psi_1\colon S\to \{0,1\}^*$ be the map that sends $s_0$ to $10$ and all other $s$ to the empty word, and let $B(\psi_1)$ be the corresponding dyadic brick. Consider the full deferment $D_{B(\psi_1)} = D_{B(\psi_1)}(SV_G)$ of $SV_G$, so $D_{B(\psi_1)}$ consists of all elements whose labeled support lies in $B(\psi_1)$. Since $B(\psi_0)$ and $B(\psi_1)$ are disjoint, we know $D_{B(\psi_1)}\le \Stab_{SV_G}(\vec{h})$. We now claim that the pair $(\Stab_{SV_G}(\vec{h}), D_{B(\psi_1)})$ satisfies the conditions of Citation~\ref{cit:coame-to-wreath}, hence $\Stab_{SV_G}(\vec{h})$ is boundedly acyclic.

Let us first prove that every finite subset $F$ of $\Stab_{SV_G}(\vec{h})$ can be conjugated by an element of $\Stab_{SV_G}(\vec{h})$ into $D_{B(\psi_1)}$. Let $\LSupp(F) \coloneqq \bigcup_{f\in F} \LSupp(f)$, so $\LSupp(F)$ is a closed subspace of $C^S$ properly contained in $C^S \setminus \{ h_0(\kappa_0), \dots, h_n(\kappa_0)\}$. We can now pick clopen subspaces $A,B\subseteq C^S$ contained in $C^S \setminus \{ h_0(\kappa_0), \dots, h_n(\kappa_0)\}$ such that $\LSupp(F) \subseteq B\subsetneq A$ and $B(\psi_1)\subsetneq A$. Now choose an element $f' \in SV$ with support in $A$ (hence $f' \in \Stab_{SV_G}(\vec{h})$) such that $f'(B)\subseteq B(\psi_1)$. This means $f'(\LSupp(F))\subseteq B(\psi_1)$ and so $F^{f'} \subseteq D_{B(\psi_1)}$.

It remains to find an element $f\in \Stab_{SV_G}(\vec{h})$ such that the conjugates of $D_{B(\psi_1)}$ by $f^p$ and $f^q$ commute for all $p\neq q$ in $\Z$. Since $D_{B(\psi_1)}$ has labeled support contained in $B(\psi_1)$, it suffices to find $f\in \Stab_{SV_G}(\vec{h})$ such that $f^p(B(\psi_1)) \cap f^q(B(\psi_1)) = \emptyset$ for all $p\neq q$. First define an element $f_{s_0}$ in $s_0V$ as the following self-homeomorphism of $\{0,1\}^\N$:
\[
f_{s_0}(u\omega)  =\begin{cases}
    00\omega   &  \text{ if } u=00,\\
    010\omega  &  \text{ if } u=01,\\
    011\omega   &  \text{ if } u=10,\\
    1\omega   &  \text{ if } u=11.
\end{cases}
\]
Now let $f\in SV$ be the element defined by $f(\kappa)(s)=\kappa(s)$ for all $s\ne s_0$ and $f(\kappa)(s_0)=f_{s_0}(\kappa(s_0))$, so intuitively $f$ ``does $f_{s_0}$ in the $s_0$ coordinate'' and does nothing in the others. Since $f$ acts on $B(\psi_0)$ as the identity and its germinal twists are $1$ everywhere, we have $B(\psi_0) \cap \LSupp(f)=\emptyset$. In particular $f\in \Stab_{SV_G}(\vec{h})$. By construction $f^p(B(\psi_1)) \cap f^q(B(\psi_1)) = \emptyset$ for all $p\neq q$, and we are done.
\end{proof}

\begin{remark}
In a recent preprint \cite{palmerwu}, Palmer and the second author prove that all labeled Thompson groups and all faithful twisted Brin--Thompson groups are acyclic (over $\Z$). It seems likely that the proof in \cite{palmerwu} can be adapted to generalize these two results and show that all abstract twisted Brin--Thompson groups are acyclic. We will not pursue this here, and refer the reader to \cite{palmerwu} for more details.
\end{remark}

\subsection{$\ell^2$-invisibility}\label{ssec:l2}

A (discrete) group is called \emph{$\ell^2$-invisible} if its homology with coefficients in its group von Neumann algebra $\mathcal{N}(G)$ vanishes in all degrees. If $\Gamma$ is of type $\F_\infty$ this is equivalent to the vanishing of homology with coefficients in $\ell^2(\Gamma)$ \cite[Lemmas 6.98 and 12.3]{lueck02}. It is a major open question, related to the \emph{zero in the spectrum conjecture}, whether there exist $\ell^2$-invisible groups of type $\F$ (meaning with finite classifying space) \cite[Chapter 12]{lueck02}. In this direction, Sauer--Thumann proved that there exist $\ell^2$-invisible groups of type $\F_\infty$, such as Thompson's group $V$ \cite{sauer:thumann}. This was then generalized by Thumann to a larger class of Thompson-like groups, which includes the Brin--Thompson groups $nV$ \cite{thumann:l2}.

In this subsection we prove that all abstract twisted Brin--Thompson groups $SV_G$ are $\ell^2$-invisible. This was known in a particularly easy case, namely when $G$ is non-amenable and $S$ is a point \cite[Theorem~7.1]{wu25}; it is new and interesting even in the faithful case. To do this, we introduce a new dynamical criterion (Theorem~\ref{thrm:general_l2_criterion}) for a group to be $\ell^2$-invisible.

\begin{definition}[Deferment system]
Let $\Gamma$ be a group acting by homeomorphisms on a space $X$. For each open subset $U\subseteq X$, let $\Delta(U)\le \Gamma$ be a subgroup. Call $\Delta\colon U\mapsto \Delta(U)$ a \emph{deferment system} if the following hold:
\begin{itemize}
    \item The support of $\Delta(U)$ is contained in $U$ for each $U$.
    \item If $U\subseteq U'$ then $\Delta(U)\le \Delta(U')$.
    \item $\Delta(g(U))=g\Delta(U)g^{-1}$ for all $U$ and $g\in \Gamma$.
    \item $\langle \Delta(U_1),\dots,\Delta(U_k)\rangle$ is a direct product $\Delta(U_1)\times\cdots\times\Delta(U_k)$ whenever $U_1,\dots,U_k$ are pairwise disjoint.
\end{itemize}
\end{definition}

Here recall that the \emph{support} of a subgroup of $\Gamma$ is the set of points in $X$ moved by some element of the subgroup. Note that if $X$ is Hausdorff (and not a single point) then necessarily $\Delta(\emptyset)=\{1\}$, since $\Delta(\emptyset)\le \Delta(U)\cap\Delta(U')=\{1\}$ for disjoint $U,U'\ne\emptyset$. Thus even though $\Gamma\curvearrowright X$ need not be faithful, the kernel is not a major player in deferment systems. On the other hand, note that if $g$ acts trivially on $X$ then $g$ must normalize every $\Delta(U)$, so the kernel is not completely invisible to deferment systems.

The quintessential example occurs in the faithful case. Recall that given a group $\Gamma$ acting on a set $X$ the \emph{rigid stabilizer} in $\Gamma$ of a subset $Y$ is
\[
\RStab_\Gamma(Y)=\{g\in\Gamma \mid g.x = x \text{ for all } x \in X\setminus Y\}\text{.}
\]

\begin{example}\label{ex:deferment:rstab}
If $\Gamma\curvearrowright X$ is faithful, then $U\mapsto \RStab_\Gamma(U)=\{g\in\Gamma\mid g.x=x$ for all $x\in X\setminus U\}$ is a deferment system.
\end{example}

\begin{theorem}[$\ell^2$-invisibility criterion]\label{thrm:general_l2_criterion}
Let $\Gamma$ be a group acting by homeomorphisms on a perfect Hausdorff space $X$. If $\Gamma\curvearrowright X$ admits a deferment system $\Delta$ such that $\Delta(U)$ is non-amenable for all $U\ne\emptyset$, then $\Gamma$ is $\ell^2$-invisible.
\end{theorem}

In order to prove this theorem, we first need to set up the isotropy spectral sequence we will use, constructed by Brown in \cite[Chapter VII.7]{brown:book}. Let $\Gamma$ be a group acting on a simplicial complex $Z$. If $\sigma$ is a simplex of $Z$, denote by $\Gamma_\sigma$ the isotropy subgroup, i.e., the setwise stabilizer of $\sigma$. Let $\chi_\sigma \colon \Gamma_\sigma \to \{ \pm 1 \}$ be the homomorphism recording the sign of the permutation induced on the vertices of $\sigma$. For a $\Z[\Gamma]$-module $M$, we denote by $M_\sigma$ the $\Z[\Gamma_\sigma]$-module whose underlying abelian group is $M$ with action defined by $g \cdot_\sigma m = \chi_\sigma(g) (g \cdot m)$.

Let $\Sigma_p$ be a choice of representatives of the $\Gamma$-orbits for the action of $\Gamma$ on the $p$-simplices of $Z$. Then there is a spectral sequence $E^k_{p, q}$ such that
\[E^1_{p, q} = \bigoplus_{\sigma \in \Sigma_p} H_q(\Gamma_\sigma; M_\sigma) \Rightarrow H_{p+q}^\Gamma(Z; M).\]
The rightmost term denotes the $\Gamma$-equivariant homology of $Z$ with coefficients in $M$.

\begin{lemma}\label{lem:spectral}
Suppose that $Z$ is acyclic, and that for every simplex $\sigma$ the action of $\Gamma_\sigma$ on $\sigma$ is trivial. Then there is a spectral sequence $E^k_{p, q}$ such that
\[E^1_{p, q} = \bigoplus_{\sigma \in \Sigma_p} H_q(\Gamma_\sigma; M) \Rightarrow H_{p+q}(\Gamma; M).\]
In particular, if $H_q(\Gamma_\sigma; M) = 0$ for all $\sigma \in \Sigma_p$ whenever $p+q \leq n$, then $H_i(\Gamma; M) = 0$ for all $i \leq n$.
\end{lemma}

\begin{proof}
Since the action of $\Gamma_\sigma$ on $\sigma$ is trivial, $\chi_\sigma$ is trivial, and so $M_\sigma$ equals $M$ as a restricted $\Z[\Gamma_\sigma]$-module. Since $Z$ is acyclic, $H_{p+q}^\Gamma(Z; M) = H_{p+q}(\Gamma; M)$. Under the vanishing assumption, the spectral sequence collapses at the first page in the relevant range of degrees.
\end{proof}

Now we are ready to prove Theorem~\ref{thrm:general_l2_criterion}. The idea of the proof is similar to that of the $\ell^2$-invisibility theorem from \cite{sauer:thumann}, except that we use a different complex that simplifies both the proof and the assumptions.

\begin{proof}[Proof of Theorem~\ref{thrm:general_l2_criterion}]
We have a perfect Hausdorff space $X$ and a group $\Gamma$ acting on $X$ by homeomorphisms. Assume the hypotheses of Theorem~\ref{thrm:general_l2_criterion} hold, namely, there is a deferment system $\Delta$ such that $\Delta(U)$ is non-amenable for every non-empty open $U \subseteq X$.

We fix $n \geq 1$ and define a directed graph $Z(n)$ as follows. The vertices are ordered tuples $\mathcal{U} = (U_1,\dots,U_k)$ of $k \ge n$ non-empty pairwise disjoint open subsets $U_i \subset M$. Since $X$ is perfect and Hausdorff, $Z(n)$ is non-empty. The directed edges are given by $\mathcal{U} \to \mathcal{V}$ if, writing $\mathcal{U}=(U_1,\dots,U_k)$ and $\mathcal{V}=(V_1,\dots,V_\ell)$, we have that each $U_i$ properly contains at least one $V_j$, and each $V_j$ is either properly contained in some $U_i$ or is disjoint from all $U_i$. We think of the edge relation informally as a ``refinement'', though note that it is usually not transitive.

We claim that for any finite set of vertices $\mathcal{U}^1, \ldots, \mathcal{U}^m$, there exists  a vertex $\mathcal{V}$ such that $\mathcal{U}^j \to \mathcal{V}$ is a directed edge, for every $j$. Indeed, write $\mathcal{U}^j = (U^j_1, \ldots, U^j_{k_j})$. For each $j \in \{1, \ldots, m\}, i \in \{ 1, \ldots, k_j\}$ choose a point $x^j_i \in U^j_i$; because $X$ is perfect we may choose these points so that they are all distinct. Choose disjoint open neighborhoods $V^j_i$ of $x^j_i$ such that $V^j_i \subsetneq U^j_i$; this is again possible because $X$ is perfect and Hausdorff. Up to making the $V^j_i$ smaller, we get that
\[\mathcal{V} = (V^1_1, \ldots, V^1_{k_1}, \ldots, V^m_1, \ldots, V^m_{k_m}) \in Z(n)\]
satisfies $\mathcal{U}^j_i \to \mathcal{V}$ for all $i, j$.

Note that there is at most one edge between any two vertices, and moreover this directed graph is acyclic, since given a vertex $\mathcal{U} = (U_1,\dots,U_k)$ it is impossible for $U_i$ to properly contain some $U_j$. Therefore we can form the directed flag complex $Z(n)_\bullet$, i.e., the simplicial complex with vertex set equal to the vertex set of $Z(n)$, and $p$-simplices $\sigma \in Z(n)_p$ of the form $\sigma = \{ \mathcal{U}^0, \ldots,\mathcal{U}^p \mid \mathcal{U}^i \to \mathcal{U}^j \text{ for all } i < j\}$. Call $\mathcal{U}^p$ the \emph{top} vertex of this simplex. By the claim above, every finite subcomplex of $Z(n)_\bullet$ is contained in a cone, and since simplicial cones are contractible this implies that $Z(n)_\bullet$ is contractible hence acyclic.

The group $\Gamma$ acts on the vertex set of $Z(n)$ via $g.(U_1,\dots,U_m)=(g(U_1),\dots,g(U_m))$. Since this action respects the edge relation, it extends to an action on $Z(n)_\bullet$. Since the graph is acyclic, for every simplex $\sigma$ the action of $\Gamma_\sigma$ on $\sigma$ is trivial. Because the tuples are ordered, $g$ fixes the vertex $(U_1, \ldots, U_m)$ if and only if $g(U_i)=U_i$ for all $1\le i\le m$.

Fix a simplex $\sigma$ of $Z(n)_\bullet$, say $\sigma = \{ \mathcal{U}^0, \mathcal{U}^1, \dots, \mathcal{U}^p \}$ with top vertex $\mathcal{U}^p$, and write $\mathcal{U}^p = (U_1, \ldots, U_k)$, so $k \geq n$. Consider the subgroup $\Lambda\le\Gamma$ defined by
\[
\Lambda \coloneqq \langle\Delta(U_1),\dots,\Delta(U_k)\rangle\text{.}
\]
Note that for any other $\mathcal{U}^r$ ($0\le r<p$), say equal to $(U_1',\dots,U_\ell')$, since $\mathcal{U}^r \to \mathcal{U}^p$ we know that each $U_i$ is either properly contained in some $U_j'$ or is disjoint from all $U_j'$. Since $\Delta(U_i)$ is supported on $U_i$, this implies that $\Delta(U_i)$ fixes $\mathcal{U}^r$. Since this holds for all $r$, we conclude that $\Delta(U_i)$ fixes $\sigma$. At this point we know that $\Lambda\le \Gamma_\sigma$. Moreover, for $g\in \Gamma_\sigma$ we know $g(U_i)=U_i$ for each $i$, so $g\Delta(U_i)g^{-1}=\Delta(g(U_i))=\Delta(U_i)$ for each $i$, and we conclude that $\Lambda$ is normal in $\Gamma_\sigma$. Since $\Delta$ is a deferment system, we have
\[
\Lambda \cong \Delta(U_1)\times\cdots\times\Delta(U_k)\text{.}
\]
Since each $\Delta(U_j)$ is non-amenable, \cite[Lemma~6.36]{lueck02} says $H_0(\Delta(U_j);\vN(\Delta(U_j)))=0$ for each $1\le j\le k$, and hence $H_i(\Lambda,\vN(\Lambda))=0$ for all $i<k$ by \cite[Lemma~12.11(3)]{lueck02}. Since $\Lambda$ is normal in $\Gamma_\sigma$, \cite[Lemma~12.11(2)]{lueck02} says $H_i(\Gamma_\sigma,\vN(\Gamma_\sigma))=0$ for all $i<k$. Finally, by \cite[Theorem~6.29]{lueck02} we conclude that $H_i(\Gamma_\sigma,\vN(\Gamma))=0$ for all $i<k$, and so in particular for all $i<n$. Now Lemma~\ref{lem:spectral} applies and we conclude that $H_i(\Gamma; \vN(\Gamma)) = 0$ for all $i < n$. Since $n$ was arbitrary, the result follows.
\end{proof}

The $\ell^2$-invisibility of abstract twisted Brin--Thompson groups follows quickly:

\begin{corollary}\label{cor:tBT_l2_invis}
Let $G$ be a group acting on a non-empty set $S$. Then the abstract twisted Brin--Thompson group $SV_G$ is $\ell^2$-invisible. In particular, every finitely generated group (quasi-isometrically) embeds in a finitely generated simple $\ell^2$-invisible group.
\end{corollary}

\begin{proof}
Consider the action of $SV_G$ on the perfect Hausdorff space $C^S$. By Theorem~\ref{thrm:general_l2_criterion} it suffices to find a deferment system $\Delta$ for the action such that $\Delta(U)$ is non-amenable for all open $U\ne\emptyset$. If the action is not faithful then we cannot use rigid stabilizers as in Example~\ref{ex:deferment:rstab}, as the kernel of the action lies in every rigid stabilizer and hennce we do not have the desired direct product decomposition. Instead we will use the full deferment $D_U=D_U(SV_G)$ of $SV_G$ to $U$. Since every such $U$ contains a dyadic brick $B(\psi)$, conjugating by the canonical homeomorphism $h_\psi$ gives us a copy of $SV_G$ inside $D_U$, so the $D_U$ are all non-amenable. We just have to verify that $\Delta(U) = D_U$ really defines a deferment system. By construction we immediately see that $D_U$ is supported on $U$, and if $U\subseteq U'$ then $D_U\le D_{U'}$. To see that $D_{f(U)}=f D_U f^{-1}$ for all $f\in SV_G$ and all $U$, note that the condition on fixing points is immediate and the condition on germinal twists follows from Equation~\eqref{eq:twist_two}. Finally, for $U_1,\dots,U_k$ pairwise disjoint, the $D_{U_i}$ pairwise commute and each $D_{U_i}$ intersects trivially with the subgroup generated by the other $D_{U_j}$, so $\langle D_{U_1},\dots,D_{U_k}\rangle\cong D_{U_1}\times\cdots\times D_{U_k}$ as desired.
\end{proof}

Let us also record the following criterion for groups acting faithfully by homeomorphisms.

\begin{corollary}[Theorem~\ref{thrm:criterion:intro}]
\label{cor:l2:faithful}
Let $\Gamma$ be a group acting faithfully by homeomorphisms on a Hausdorff space $X$. Suppose that for every non-empty open set $U$, the rigid stabilizer $\RStab_{\Gamma}(U)$ is non-amenable. Then $\Gamma$ is $\ell^2$-invisible.
\end{corollary}

\begin{proof}
The rigid stabilizer of a point is trivial, hence the assumption implies that a single point is not open so $X$ is perfect. Now the statement is a combination of Theorem~\ref{thrm:general_l2_criterion} and Example~\ref{ex:deferment:rstab}.
\end{proof}

This implies $\ell^2$-invisibility of a plethora of groups, such as homeomorphism and diffeomorphism groups of positive-dimensional manifolds in all regularities. Moreover, we obtain an intriguing characterization of amenability of $F$.

\begin{corollary}\label{cor:F_l2}
Thompson's group $F$ is non-amenable if and only if it is $\ell^2$-invisible.
\end{corollary}

\begin{proof}
The ``if'' direction is immediate by \cite[Lemma~6.36]{lueck02}. The ``only if'' direction follows from Corollary~\ref{cor:l2:faithful} together with the fact that rigid stabilizers in $F$ contain copies of $F$.
\end{proof}

Corollary~\ref{cor:l2:faithful} is a direct analog of \cite[Corollary~1.3]{leboudec:simplicity:homeo}, which says that these same hypotheses imply $C^*$-simplicity (in fact $C^*$-simplicity will be the focus of the next section). For Thompson's group $F$, Corollary~\ref{cor:F_l2} is a partial analog of \cite[Theorem~1.6]{leboudec:simplicity:homeo}, which says $F$ is non-amenable if and only if it is $C^*$-simple, if and only if $T$ is $C^*$-simple. To try and tie this to the $\ell^2$-invisibility of $T$, note that a similar proof looking at $T$ acting on $S^1$ shows that if $F$ is non-amenable then $T$ is $\ell^2$-invisible. We do not know whether $T$ being $\ell^2$-invisible is sufficient for $F$ to be non-amenable.

\section{$C^*$-simplicity}\label{sec:Cstar}

In this section we prove that an abstract twisted Brin--Thompson group $SV_G$ is $C^*$-simple if and only if it has trivial amenable radical, and characterize precisely when this happens in terms of the action $G\curvearrowright S$:

\begin{theorem}\label{thrm:tbt_C*_simple}
Let $G$ be a group acting on a non-empty set $S$ with kernel $K$. The following are equivalent.
\begin{enumerate}
    \item $SV_G$ is $C^*$-simple;
    \item $SV_G$ has trivial amenable radical;
    \item There is no non-trivial amenable normal subgroup of $G$ contained in $K$;
    \item $K$ has trivial amenable radical.
\end{enumerate}
\end{theorem}

Recall that a (discrete) group is called \emph{$C^*$-simple} if its reduced $C^*$-algebra is simple. We will not introduce this analytic point of view, and rather use the following purely group-theoretic characterization. Here a subgroup $\Lambda$ of a group $\Gamma$ is called \emph{confined} if there exists a finite subset $F \subset \Gamma \setminus \{ 1 \}$ such that $gFg^{-1} \cap \Lambda \neq \emptyset$ for every $g \in \Gamma$. Note that the trivial subgroup is not confined.

\begin{cit}\cite{kennedy}\label{cit:kennedy}
A group $\Gamma$ is $C^*$-simple if and only if it has no confined amenable subgroups.
\end{cit}

The original reference \cite{kennedy} calls confined subgroups \emph{residually normal}, although the name \emph{confined} had been established much earlier \cite{confined} and is now the standard term for this notion, see, e.g., \cite{confined:LBMB}.

If $\Gamma$ has some non-trivial amenable normal subgroup, equivalently if $\Gamma$ has non-trivial amenable radical (defined to be the unique maximal amenable normal subgroup), then of course it cannot be $C^*$-simple. On the other hand, there exist finitely generated groups with trivial amenable radical that are not $C^*$-simple \cite{leboudec:simplicity}. Thus, Theorem~\ref{thrm:tbt_C*_simple} says that such examples cannot happen among abstract twisted Brin--Thompson groups.

\medskip

Let us spell out the following easy rephrasing of the relevant definitions; for a subset $F\subseteq \Gamma\setminus\{1\}$ and subgroup $\Lambda\le\Gamma$, say that $F$ is \emph{displaceable} from $\Lambda$ if there exists $g\in\Gamma$ such that $gFg^{-1}\cap\Lambda=\emptyset$.

\begin{observation}\label{obs:displace}
Let $\Gamma$ be a group. Then:
\begin{itemize}
    \item $\Gamma$ is $C^*$-simple if and only if every finite $F\subseteq \Gamma\setminus\{1\}$ is displaceable from every amenable $\Lambda\le\Gamma$.
    \item $\Gamma$ has trivial amenable radical if and only if every singleton $\{f\}\subseteq \Gamma\setminus\{1\}$ is displaceable from every amenable $\Lambda\le\Gamma$.
\end{itemize}
\end{observation}

\begin{proof}
The first item is just Citation~\ref{cit:kennedy} rephrased. For the second item, if $\Gamma$ has non-trivial amenable radical then a non-trivial element of the amenable radical cannot be displaced from the amenable radical, and conversely if some non-trivial element cannot be displaced from some amenable subgroup, then the normal core of that subgroup is a non-trivial amenable normal subgroup.
\end{proof}

\medskip

We first focus on the faithful case.

\begin{lemma}\label{lem:rstab}
Let $G \curvearrowright S$ be a faithful action on a non-empty set, and consider the induced action of $SV_G$ on $C^S$. Then, for every non-empty open subset $U \subset C^S$, the rigid stabilizer $\RStab_{SV_G}(U)$ is non-amenable.
\end{lemma}

\begin{proof}
In the faithful case, the deferment system from the proof of Corollary~\ref{cor:tBT_l2_invis} coincides with that of rigid stabilizers from Example~\ref{ex:deferment:rstab}. Non-amenability was shown in the proof of Corollary~\ref{cor:tBT_l2_invis}.
\end{proof}

This is exactly the sufficient condition that Le Boudec--Matte Bon show in \cite[Corollary~1.3]{leboudec:simplicity:homeo} implies $C^*$-simplicity, hence we immediately deduce:

\begin{corollary}\label{cor:faithful_C*}
    If $G \curvearrowright S$ is a faithful action of a group on a non-empty set, then $SV_G$ is $C^*$-simple.\qed
\end{corollary}

Although the results of \cite{leboudec:simplicity:homeo} are stated under a countability assumption, Adrien Le Boudec has informed us that this hypothesis is never used in the paper. Accordingly, we state Corollary~\ref{cor:faithful_C*} without a countability assumption.

Now let us consider the general case. Let $G \curvearrowright S$ be an action on a non-empty set with kernel $K$. Slightly generalizing Equation~\eqref{eq:KT}, for a subgroup $L < K$ and a multicolored tree $T$ with $n$ leaves, we define
\[L_T \coloneqq \{[T, \id, (l_1, \ldots, l_n), T] \mid l_1, \ldots, l_n \in L \} \cong L^n,\]
and let $SL_G$ be the directed union of the $L_T$ over all multicolored trees $T$. If $L$ is normal in $G$, then $SL_G$ is the kernel of the quotient map $SV_G \to SV_{G/L}$.

\medskip

Now we can use the result in the faithful case to prove our main result of the section.

\begin{proof}[Proof of Theorem~\ref{thrm:tbt_C*_simple}]
Let us start by noticing that, since $K$ is normal in $G$, and the amenable radical is characteristic, $K$ has non-trivial amenable radical if and only if there exists an amenable normal subgroup $L$ of $G$ contained in $K$. In this case, $SL_G$ is a non-trivial amenable normal subgroup of $SV_G$, so $SV_G$ has non-trivial amenable radical. As we mentioned above, $C^*$-simple groups have trivial amenable radical. This proves the implications (i) $\Rightarrow$ (ii) $\Rightarrow$ (iii) $\Leftrightarrow$ (iv). It remains to show that (iv) $\Rightarrow$ (i). That is, suppose $K$ has trivial amenable radical, and we have to show that $SV_G$ is $C^*$-simple. Let $\Lambda\le SV_G$ be amenable, so by Observation~\ref{obs:displace} we must prove that every finite $F \subset SV_G\setminus\{1\}$ can be displaced from $\Lambda$.

Suppose first that $\Lambda\le SK_G$. Since $SK_G$ is normal, without loss of generality $F\subseteq SK_G$. For each $f\in F$ choose $\kappa_f\in C^S$ with $\gtwist_{\kappa_f}(f)\ne 1$, with the $\kappa_f$ distinct for distinct $f$. Since $K$ acts trivially on $C^S$, the germinal twist function at any point in $C^S$ restricts to $SK_G$ as a homomorphism $SK_G\to K$. In particular $L_f\coloneqq \gtwist_{\kappa_f}(\Lambda)$ is an amenable subgroup of $K$. Since $K$ has trivial amenable radical and $\gtwist_{\kappa_f}(f)\ne 1$, we can therefore choose $k_f\in K$ such that $k_f \gtwist_{\kappa_f}(f)k_f^{-1}\not\in L_f$.  Now choose $h\in SK_G$ satisfying $\gtwist_{\kappa_f}(h)=k_f$ for all $f\in F$. We claim $h F h^{-1}\cap \Lambda=\emptyset$. Consider $hfh^{-1}$ for $f\in F$. By the construction of $h$, we have $\gtwist_{\kappa_f}(h)=k_f$, so by the choice of $k_f$ the image of $hfh^{-1}$ under $\gtwist_{\kappa_f}$ does not lie in $L_f$. We conclude that $hfh^{-1}\not\in\Lambda$, and since $f$ was arbitrary we get $h F h^{-1}\cap \Lambda=\emptyset$.

Now suppose $\Lambda$ is not contained in $SK_G$. Write $\pi$ for the quotient map $SV_G\to SV_{G/K}$. Let $F_{ker}=F\cap SK_G$ and let $F_{im}=F\setminus F_{ker}$, so $\pi(F_{im})$ is a finite subset of $SV_{G/K}\setminus\{1\}$. Since $\pi(\Lambda)$ is an amenable subgroup of $SV_{G/K}$, and $SV_{G/K}$ is $C^*$-simple by Corollary~\ref{cor:faithful_C*}, we can choose $h \in SV_G$ such that $\pi(h) \pi(F_{im}) \pi(h)^{-1}\cap \pi(\Lambda)=\emptyset$, and hence $h F_{im} h^{-1} \cap \Lambda = \emptyset$. Now $\Lambda\cap SK_G$ is an amenable subgroup of $SV_G$ that is contained in $SK_G$, so by the previous paragraph we can choose $k\in SK_G$ such that $kh F_{ker} h^{-1}k^{-1}\cap \Lambda = \emptyset$. Moreover,
\[\pi(kh F_{im} h^{-1}k^{-1}) \cap \pi(\Lambda) = \pi(h) \pi(F_{im}) \pi(h)^{-1}\cap \pi(\Lambda) = \emptyset,\]
and so $kh F_{im} h^{-1}k^{-1}\cap \Lambda = \emptyset$ as well. This shows that $kh F h^{-1}k^{-1}\cap \Lambda = \emptyset$ and concludes the proof.
\end{proof}

\begin{remark}
Focusing only on the canonical kernel $SK_G$, this proof shows in particular that $SK_G$ is $C^*$-simple as soon as $K$ has trivial amenable radical, even if $K$ is not $C^*$-simple itself. An analogous proof shows the following easier statement that is still interesting. Let $H$ be a group and for each $1\le k\le n$ let $\iota_k^n\colon H^n\to H^{n+1}$ send $(h_1,\dots,h_n)$ to $(h_1,\dots,h_k,h_k,\dots,h_n)$. Then the direct limit of the directed system of the $H^n$ with all the $\iota_k^n$ is $C^*$-simple if and only if $H$ has trivial amenable radical. This is striking since $H$ itself may not be $C^*$-simple, e.g., one of the examples of Le Boudec in \cite{leboudec:simplicity}. Hence the property of being non-$C^*$-simple is not preserved under direct limits.
\end{remark}

\appendix

\section{A strengthening of the Boone--Higman--Thompson theorem}\label{appendix}

Recall from Remark~\ref{rmk:bht_and_he} the Boone--Higman--Thompson theorem, that a finitely generated group has solvable word problem if and only if it embeds in a finitely generated simple group that embeds in a finitely presented group \cite{boone74, thompson80}. In this appendix we slightly strengthen this theorem, and also discuss connections to twisted Brin--Thompson groups and the Boone--Higman conjecture. The statement can be phrased in a number of equivalent ways: the following is left to the reader.

\begin{lemma}\label{strong_bht:equiv}
    Let $\Gamma \leq P$ be groups. Then the following are equivalent.
    \begin{itemize}
        \item Every non-identity element of $\Gamma$ normally generates $P$.
        \item Every proper normal subgroup of $P$ intersects $\Gamma$ trivially.
        \item Every non-trivial quotient of $P$ is injective on $\Gamma$. \qed
    \end{itemize}
\end{lemma}

We will need the following lemma.

\begin{lemma}\label{lem:norm_gen}
    Let $\Gamma$ be a finitely presented group and $1 \neq g \in \Gamma$. Then there exists a finitely presented group $P$ containing $\Gamma$, such that $g$ normally generates $P$.
\end{lemma}

\begin{proof}
    First embed $\Gamma$ into a finitely generated simple group $G$, for instance the twisted Brin--Thompson group $\Gamma V_\Gamma$. We claim that there exists a finitely presented cover $P \to G$ satisfying the hypotheses. Indeed, since $\Gamma$ is finitely presented, only finitely many relations from $G$ suffice to ensure that $\Gamma$ embeds. Moreover, $g$ normally generates $G$ since $G$ is simple, that is, every generator of $G$ is a product of conjugates of powers of $g$, and this can again be expressed by only finitely many relations.
\end{proof}

\begin{remark}
    Recently, Chatterji--Kassabov proved that $P$ can moreover be chosen to have property (T) \cite{chatterji26}, hence property (T) can be assumed of the container group in Proposition \ref{prop:strong_bht}. In fact, via an easy construction using small cancellation theory over acylindrically hyperbolic groups \cite{hull:sc}, the container group $P$ may be chosen to be a quotient of any given finitely presented acylindrically hyperbolic group. Since these additional properties are not relevant to us, we only need Lemma \ref{lem:norm_gen}, for which we could present a self-contained elementary proof.
\end{remark}

The main result of this appendix is the following.

\begin{proposition}\label{prop:strong_bht}
Let $\Gamma$ be a finitely generated group. Then $\Gamma$ has solvable word problem if and only if there exists a finitely presented group $P$ such that $\Gamma$ embeds as a subgroup of $P$ trivially intersecting every proper normal subgroup of $P$.
\end{proposition}

\begin{proof}
First suppose $\Gamma$ has solvable word problem, so by the Boone--Higman--Thompson theorem there exist a finitely generated simple group $G$ and a finitely presented group $P$ such that $\Gamma\le G\le P$. By Lemma \ref{lem:norm_gen}, there exists a finitely presented group $Q$ containing $P$ such that an element $1 \neq g \in G$ normally generates $Q$. Since $G$ is simple, it follows that every non-identity element in $G$, hence in particular in $\Gamma$, normally generates $Q$. We conclude by Lemma \ref{strong_bht:equiv}.

Now suppose $\Gamma$ is a subgroup of a finitely presented group $P$ such that $\Gamma$ trivially intersects every non-trivial normal subgroup of $P$. We must show that $\Gamma$ has solvable word problem.

We follow the classical Kuznetsov algorithm \cite{kuznetsov58}, which treated the case when $P$ is simple. Fix a finite presentation $\langle S \mid R \rangle$ for $P$, and fix a word $w$ in the alphabet $S$ representing an element of $\Gamma$. We run a ``day and night'' algorithm, which returns ``yes'' if $w$ represents the identity, and ``no'' otherwise.

During the day, we enumerate all relations in the presentation $\langle S \mid R \rangle$, and check whether $w$ appears as one of them. If it does, then the algorithm stops and returns ``yes''.

During the night, we enumerate all relations in the presentation $\langle S \mid R \cup \{ w \} \rangle$, and check whether each generator in $S$ appears as a relation. If they all do, then the algorithm stops and returns ``no''.

We claim that this algorithm terminates. Indeed, if $w$ represents the identity, then it must eventually appear in an enumeration of all the relations. Otherwise, $w$ represents an element in $\Gamma \setminus \{ 1 \}$, hence a normal generator of $P$, by Lemma \ref{strong_bht:equiv}. Therefore $\langle S \mid R \cup \{ w \} \rangle$ is a presentation for the trivial group. Thus, all the generators in $S$ must eventually appear in an enumeration of all the relations of this new presentation, and we are done.
\end{proof}

This immediately gives us the following strengthening of the Boone--Higman--Thompson theorem:

\begin{corollary}\label{cor:strong_bht}
A finitely generated group $\Gamma$ has solvable word problem if and only if it embeds in a finitely generated simple group $G$ that embeds in a finitely presented group $P$, such that $G$ normally generates $P$.\qed
\end{corollary}

The relative Boone--Higman conjecture amounts to saying that the group $P$ in Proposition~\ref{prop:strong_bht} can always be taken to be relatively simple, so conjecturally the statement about trivially intersecting every proper normal subgroup boils down to just trivially intersecting one proper normal subgroup (the largest). Note that if $\Gamma$ is infinite, then since it embeds in every non-trivial quotient of $P$ by Lemma \ref{strong_bht:equiv}, it follows that $P$ has no non-trivial finite quotients. Thus a Zorn's lemma argument going back to Higman \cite{higman:group} shows that $P$ admits a (finitely generated) simple quotient. If one could show that, in this situation, $P$ admits a \emph{finitely presented} simple quotient, then this would prove the Boone--Higman conjecture.

\bibliographystyle{alpha}
\newcommand{\etalchar}[1]{$^{#1}$}

\end{document}